\documentclass{article}[12pt]
\usepackage{a4}
\usepackage[english]{babel}
\usepackage[T1]{fontenc}
\usepackage[all]{xy}
\usepackage{amssymb}
\usepackage{amsmath}
\usepackage{amsthm}
\usepackage{multirow}
\usepackage{graphicx}

\newtheorem{thm}{Theorem}[section]
\newtheorem{prop}[thm]{Proposition}
\newtheorem{lem}[thm]{Lemma}
\newtheorem{coro}[thm]{Corollary}
\newtheorem{defn}[thm]{Definition}
\newtheorem{expl}[thm]{Example}
\newtheorem{rmk}[thm]{Remark} 

\newtheorem{conj}[thm]{Conjecture}
\newtheorem{nota}[thm]{Notation}
\numberwithin{equation}{section}

\newcommand{\rbb}{\mathbb{R}}

\newcommand{\lra}{\longrightarrow}

\newcommand{\mcalk}{\mathcal{K}}

\newcommand{\mcalo}{\mathcal{O}}

\newcommand{\mcalc}{\mathcal{C}}

\newcommand{\om}{\mathcal{O}(M)}

\newcommand{\okm}{\mathcal{O}_k(M)}

\newcommand{\mcala}{\mathcal{A}}
\newcommand{\mcalf}{\mathcal{F}}
\newcommand{\mcald}{\mathcal{D}}

\newcommand{\mcalm}{\mathcal{M}}

\newcommand{\mcalb}{\mathcal{B}}

\newcommand{\mcali}{\mathcal{I}}

\newcommand{\bkm}{\mathcal{B}_k(M)}

\newcommand{\bku}{\mathcal{B}^{(k)}}

\newcommand{\psit}{\widetilde{\Psi}}

\newcommand{\hra}{\hookrightarrow}
\newcommand{\lla}{\longleftarrow}
\newcommand{\ssie}{\subseteq_{ie}}

\newcommand{\ok}{\mathcal{O}_k}

\newcommand{\ftild}{\widetilde{f}}

\newcommand{\dv}{\mcald(V)}
\newcommand{\dvp}{\mcald(V')}

\newcommand{\xa}{\mathcal{X}_A}
\newcommand{\yb}{\mathcal{Y}_B}
\newcommand{\zc}{\mathcal{Z}_C}
\newcommand{\mcale}{\mathcal{E}}

\newcommand{\mcalj}{\mathcal{J}}
\newcommand{\mcalx}{\mathcal{X}}
\newcommand{\mcaly}{\mathcal{Y}}
\newcommand{\phit}{\widetilde{\Phi}}
\newcommand{\ra}{\rightarrow}
\newcommand{\bt}{\widetilde{B}}
\newcommand{\xtild}{\widetilde{X}}
\newcommand{\bk}{\mcalb_k}

\newcommand{\bkp}{\widetilde{\mcalb}_{k, p}}  
\newcommand{\bkd}{\widetilde{\mcalb}_{k, \bullet}}  
\newcommand{\okp}{\widetilde{\mcalo}_{k, p}} 
\newcommand{\fsb}{F^{!}_{\mcalb}}   
\newcommand{\fsbpr}{F^{!}_{\mcalb'}}
\newcommand{\fso}{F^{!}_{\mcalo}} 
\newcommand{\gso}{G^{!}_{\mcalo}} 
\newcommand{\fsbp}{\tilde{F}^{!p}_{\mcalb}}   
\newcommand{\fsbd}{\tilde{F}^{!\bullet}_{\mcalb}} 
\newcommand{\fsop}{\tilde{F}^{!p}_{\mcalo}} 
 
\newcommand{\intcf}{\int_{\mcalc} \mcalf} 
\newcommand{\ftb}{\tilde{F}_{\mcalb}}
\newcommand{\fto}{\tilde{F}_{\mcalo}}
\newcommand{\fhb}{\hat{F}_{\mcalb}}
\newcommand{\fho}{\hat{F}_{\mcalo}}
\newcommand{\bkq}{\widehat{\mcalb}_{k, q}}
\newcommand{\okq}{\widehat{\mcalo}_{k, q}}
\newcommand{\mcalco}{\mcalc^{\text{op}}}
\newcommand{\pid}{\Pi^{\bullet}}
\newcommand{\ftpb}{\tilde{F}^p_{\mcalb}}
\newcommand{\ftpo}{\tilde{F}^p_{\mcalo}}

\newcommand{\hpd}{\underset{[p] \in \Delta}{\text{holim}}}
\newcommand{\hqd}{\underset{[q] \in \Delta}{\text{holim}}}
\newcommand{\hbkp}{\underset{\bkp (U)}{\text{holim}}}
\newcommand{\hokp}{\underset{\okp (U)}{\text{holim}}}
\newcommand{\tot}{\text{Tot}}
\newcommand{\hbkq}{\underset{\bkq (U)}{\text{holim}}}
\newcommand{\hokq}{\underset{\okq (U)}{\text{holim}}}
\newcommand{\bkb}{\overline{\mcalb}_k}
\newcommand{\zd}{Z^{\bullet}}
\newcommand{\dar}{\downarrow}
\newcommand{\fko}{\mcalf_k(\om; \mcalm)}

\title{ \textbf{Polynomial functors in manifold calculus}}
\date{}

\author{Paul Arnaud Songhafouo Tsopm\'en\'e \\
Donald Stanley}

\begin{document}
\maketitle

\begin{abstract} 
 Let $M$ be a smooth manifold, and let $\mathcal{O}(M)$ be the poset of open subsets of $M$. Manifold calculus, due to Goodwillie and Weiss, is a calculus of functors suitable for studying contravariant functors (cofunctors)
$F \colon \om \lra \text{Spaces}$ from $\om$  to the category of spaces. Weiss showed that polynomial cofunctors of degree $\leq k$ are determined by their values on $\okm$, where $\okm$ is the full subposet of $\om$ whose objects are open subsets diffeomorphic to the disjoint union of at most $k$ balls.  Afterwards Pryor showed that one can replace $\okm$ by more general subposets and still recover the same notion of polynomial cofunctor. In this paper, we generalize these results to cofunctors from $\om$ to any simplicial model category $\mcalm$. If $F_k(M)$ stands for the unordered configuration space of $k$ points in $M$,  we also show that the category of homogeneous cofunctors $\om \lra \mcalm$ of degree $k$ is weakly equivalent to the category of linear cofunctors $\mcalo (F_k(M)) \lra \mcalm$ provided that $\mcalm$ has a zero object. Using a new approach, we also show that if $\mcalm$ is a general model category and  $F \colon \okm \lra \mcalm$ is an isotopy cofunctor, then the homotopy right Kan extension of $F$ along the inclusion $\okm \hra \om$ is also an isotopy cofunctor. 
 
 % Let  $\mcalm$ be a \lq\lq nice\rq\rq{} model category. In manifold calculus a homogeneous functor $F \colon \om \lra \mcalm$ is called \textit{very good} (respectively \textit{good}) if it sends isotopy equivalences to isomorphisms (respectively weak equivalences). In previous work we studied very good homogeneous functors of degree $k$, and we proved that the category of such functors is equivalent to that of contravariant functors from the fundamental groupoid of $F_k(M)$ to $\mcalm$, where $F_k(M)$ stands for the unordered configuration space of $k$ points in $M$. In this paper we generalize this result to good homogeneous functors. 
\end{abstract}

\tableofcontents

\section{Introduction}

%This work deals with manifold calculus due to   Goodwillie and Weiss in  \cite{good_weiss99, wei99}. The paper is subdivided into two  detailed and almost disconnected parts. Part I generalizes the works of Weiss \cite{wei99} and Pryor \cite{pryor15}. Specifically, that part characterizes polynomial cofunctors \footnote{In this paper the word \lq\lq cofunctor\rq\rq{} means contravariant functor} in any simplicial model category. Part II treats, with a new approach, isotopy cofunctors in general model categories. Before we state our results, we will first recall what is done in \cite{wei99} and \cite{pryor15}. 

 Let $M$ be a smooth manifold, and let $\mathcal{O}(M)$ be the poset of open subsets of $M$. \textit{Manifold calculus} is a calculus of functors suitable for studying cofunctors \footnote{In this paper the word \lq\lq cofunctor\rq\rq{} means contravariant functor}  $F : \om \lra \text{Spaces}$ from $\om$ to the category of spaces (of which the embedding functor $\text{Emb}(-, W)$ for a fixed manifold $W$ is a prime example). So manifold calculus belongs to the world of calculus of functors, and therefore it definitely has a notion of polynomial cofunctor. Roughly speaking, a \textit{polynomial cofunctor} is a contravariant functor $\om \lra \text{Spaces}$ that satisfies an appropriate higher-order excision property, similar to the case of \cite{good03} (see Definition~\ref{poly_defn}). In \cite[Theorems 4.1, 5.1]{wei99} Weiss characterizes polynomial cofunctors. More precisely, he shows that polynomial cofunctors  of degree $\leq k$ are determined (up to equivalence of course) by their values on $\okm$, the full subposet of $\om$ whose objects are open subsets diffeomorphic to the disjoint union of at most $k$ balls. Many examples of polynomial and homogeneous cofunctors are also provided in \cite{wei99}.  Another good reference where the reader can find an introduction to manifold calculus is \cite{mun10}.

Weiss' characterization of polynomial cofunctors was generalized by Pryor in \cite{pryor15} as follows. Let $\mcalb$ be a basis  for the topology of $M$. We assume that $\mcalb$ is \textit{good}, that is, every element of $\mcalb$ is a subset of $M$ diffeomorphic to an open ball. For instance, if $M = \rbb^m$, we can take $\mcalb$ to be the collection of genuine open balls (with respect to the euclidean metric), or cubes, or simplices, or convex $d$-bodies more generally. For $k \geq 0$, we let $\bkm \subseteq \okm$ denote the full subposet whose objects are disjoint unions of at most $k$ elements from $\mcalb$.  So one possible choice of $\bkm$ is $\okm$ itself. In \cite[Theorem 6.12]{pryor15} Pryor shows, in the same spirit as Weiss, that any polynomial cofunctor $\om \lra \text{Spaces}$ of degree $\leq k$ is determined by its restriction to $\bkm$.  So one can replace $\okm$ by $\bkm$ without losing any homotopy theoretic information when forming the polynomial approximation to a cofunctor.

In this paper we generalize the aforementioned results of Weiss-Pryor  to cofunctors from $\om$ to any simplicial model category $\mcalm$. Specifically we have the following theorem, which is our first result.

\begin{thm} \label{main_thm}
Let $\mcalb$ be a good basis (see Definition~\ref{gb_defn}) for the topology of $M$, and let $\bkm \subseteq \om$ be the full subposet whose objects are  disjoint unions of at most $k$ elements from $\mcalb$. Consider a simplicial model category $\mcalm$ and a  cofunctor $F \colon \om \lra \mcalm$.  Then $F$  is good (see Definition~\ref{good_defn}) and polynomial of degree $\leq k$ (see Definition~\ref{poly_defn}) if and only if the restriction $F|\bk(M)$ is an isotopy cofunctor (see Definition~\ref{isotopy_cof_defn}) and the canonical map $F \lra (F|\bk(M))^!$ is a weak equivalence. Here $(F|\bk(M))^! \colon \om \lra \mcalm$ is the cofunctor defined as  
\[
(F|\bk(M))^!(U):= \underset{V \in \bk(U)}{\text{holim}} \; F(V).
\] 
%there exist an isotopy cofunctor $G \colon \mcalb_k(M) \lra \mcalm$ (see Definition~\ref{isotopy_cof_defn}) and a weak equivalence $F \stackrel{\sim}{\lra} G^{!}_{\mcalb}$. Here $G^{!}_{\mcalb} \colon \om \lra \mcalm$ is defined as 
%\[
%G^{!}_{\mcalb} (U) = \underset{V \in \bk(U)}{\text{holim}} \; G(V). 
%\]
\end{thm}

Notice that Theorem~\ref{main_thm} implies that the category of good polynomial cofunctors $\om \lra \mcalm$ of degree $\leq k$ is \textit{weakly equivalent}, in the sense of Definition~\ref{we_defn}, to the category of isotopy cofunctors $\bk(M) \lra \mcalm$. Also notice that our definition of \textit{good cofunctor} is slightly different from the classical one (see \cite[Page 71]{wei99} or \cite[Definition 1.3.4]{mun10} or \cite[Definition 3.1]{pryor15}) as we add an extra axiom: our cofunctors are required to be also objectwise fibrant. We need that extra axiom to be able to use the homotopy invariance theorem (see Theorem~\ref{fib_cofib_thm}) and the cofinality result (see Theorem~\ref{htpy_cofinal_thm}). 
%(Those results together with the Fubini theorem are the most important properties for homotopy limits we need.) 
If one works with a category $\mcalm$ in which every object is fibrant, the extra axiom  becomes a tautology. This is the case  in  Weiss' paper \cite{wei99} where  $\mcalm =$ Spaces.  For the main ingredients in the proof of Theorem~\ref{main_thm}, see  \lq\lq Outline of the paper\rq\rq{} below. 

As mentioned earlier, our result generalizes those of Weiss. In fact, from Theorem~\ref{main_thm} with $\mcalm=\text{Spaces}$ and $\mcalb = \mcalo$, the maximal good basis, one can easily deduce the main results of \cite{wei99}, which are Theorems 4.1, 5.1 and 6.1. 

The following conjecture says that Theorem~\ref{main_thm} still holds when $\mcalm$ is replaced by a general model category. We believe in that conjecture, which could  be handled by using the same approach as that we use to show Theorem~\ref{main_thm}. The  issue with that approach is the fact that some important results/properties regarding  homotopy limits  in a general model category (for example Theorem~\ref{holim_tot_thm} and Proposition~\ref{comm_prop}) are available nowhere in the literature.  So the proof of  the conjecture may turn into a matter of homotopy limits. A good reference, where the reader can find the definition and several useful properties of homotopy limits (in a general model category of course), is  \cite[Chapter 19]{hir03}. Another good reference is \cite{dhks04}. 

\begin{conj} \label{main_conj}
Theorem~\ref{main_thm} remains true if one replaces $\mcalm$ by a general model category.  
\end{conj}

Now we state our second result. Given a good cofunctor $F \colon \om \lra \mcalm$ one can define its $k$th polynomial approximation, denoted $T_kF \colon \om \lra \mcalm$, as the homotopy right Kan extension of the restriction $F|\okm$ along the inclusion $\okm \hra \om$.  In order words $T_kF(U) := \underset{V \in \okm}{\text{holim}} \; F(V)$. The \lq\lq difference\rq\rq{} between $T_kF$ and $T_{k-1}F$ belongs to a nice class of cofunctors called \textit{homogeneous cofunctors} of degree $k$ (see Definition~\ref{hc_defn}). When $k =1$ we talk about \textit{linear cofunctors}. Thanks to the fact that Theorem~\ref{main_thm} holds for any good basis $\mcalb$ we choose, one can prove the following result, which  roughly states that the category of homogeneous cofunctors $\om \lra \mcalm$ of degree $k$ is weakly equivalent to the category of linear cofunctors $\mcalo (F_k(M)) \lra \mcalm$. Here $F_k(M)$ stands for  the unordered configuration space of $k$ points in $M$.

\begin{thm} \label{main2_thm}
Let $\mcalm$ be a simplicial model category. Assume that $\mcalm$ has a zero object (that is, an object which is both terminal an initial).
\begin{enumerate}
	\item[(i)] Then the category $\fko$ of homogeneous cofunctors $\om \lra  \mcalm$ of degree $k$ (see Definition~\ref{hc_defn}) is weakly equivalent (in the sense of Definition~\ref{we_defn}) to the category of linear cofunctors $\mcalo (F_k(M)) \lra \mcalm$. That is,
	\[
	\mcalf_k(\om; \mcalm) \simeq \mcalf_1 (\mcalo (F_k(M)); \mcalm). 
	\]
   \item[(ii)]   For $A \in \mcalm$ we have the weak equivalence
   \[
   \mcalf_{kA}(\om; \mcalm) \simeq \mcalf_{1A} (\mcalo (F_k(M)); \mcalm),
   \]   
   where $\mcalf_{kA}(\om; \mcalm)$  stands for the category of homogeneous cofunctors $F \colon \om \lra \mcalm$ of degree $k$ such that $F(U) \simeq  A$ for any $U$ diffeomorphic to the disjoint union of exactly $k$ open balls. 
\end{enumerate}
 
\end{thm}

The second part of this result will be used in \cite{paul_don17-3}. A similar result (but with a different approach) to Theorem~\ref{main2_thm} was obtained by the authors in \cite[Corollary 3.31]{paul_don17}  for \lq\lq very good homogeneous functors\rq\rq{}. Note that neither \cite[Corollary 3.31]{paul_don17} nor Theorem~\ref{main2_thm} was known before, even for $\mcalm = \text{Spaces}$. Theorem~\ref{main2_thm} is interesting in the sense that it reduces the study of homogeneous cofunctors of degree $k$ to the study of linear cofunctors, which are easier to handle. In \cite{paul_don17-3} we use it (Theorem~\ref{main2_thm}) as the starting point in the classification of homogeneous cofunctors of degree $k$.

Our third result is a partial answer to Conjecture~\ref{main_conj}.  

\begin{thm} \label{iso_cof_thm}
Let $\mcalm$ be a  model category.  Let $F \colon \okm \lra \mcalm$ be an isotopy cofunctor (see Definition~\ref{isotopy_cof_defn}). Then the cofunctor $F^{!} \colon \om \lra \mcalm$ defined as 
\begin{eqnarray} \label{fsrik_defn}
F^{!}(U) := \underset{V \in \ok(U)}{\text{holim}} \; F(V)
\end{eqnarray}
 is an isotopy cofunctor as well. 
\end{thm}

The method we use to prove Theorem~\ref{iso_cof_thm} is completely different from that we use to prove Theorem~\ref{main_thm} essentially because of the following. First note that in Theorem~\ref{main_thm}  $\mcalm$ is a simplicial model category, while in Theorem~\ref{iso_cof_thm} $\mcalm$ is  a general model category. To prove Theorem~\ref{main_thm}, we use several results/properties of homotopy limits in simplicial model categories such as the Fubini theorem (see Theorem~\ref{fubini_thm}), and Proposition~\ref{comm_prop}. However, Proposition~\ref{comm_prop} involves the notion of totalization of a cosimplicial object, which a priori does not make sense in a general model category. 

The key concept we introduce to prove Theorem~\ref{iso_cof_thm} is called \textit{admissible family} of open subsets. Roughly speaking, a sequence  $B = B_0, \cdots, B_n$ of open balls  is said to be \textit{admissible} if $B_i \cap B_{i+1} \neq \emptyset$ for all $i$. One can extend that definition to sequences $V=V_0, \cdots, V_n$ of objects of $\okm$ (see Definition~\ref{pw_adm_defn}). Such sequences yield  zigzags of isotopy equivalences of $\okm$ between $V$ and $V_n$, and the collection of those  form a category denoted $\dv$ (see Definition~\ref{dv_defn}). This latter category plays a crucial role in Section~\ref{iso_cof_section}. Indeed, one can \lq\lq deduce\rq\rq{} Theorem~\ref{iso_cof_thm} by applying the homotopy limit functor to appropriate diagrams in $\mcalm$ indexed by $\dv$.

\textbf{Outline of the paper} 

This paper is subdivided into two detailed and almost disconnected parts. The first one covers Sections~\ref{notation_section}, \ref{holim_simpl_section}, \ref{sos_good_section}, \ref{poly_section} and \ref{hc_section} where we prove Theorems~\ref{main_thm} and \ref{main2_thm}, while the second covers Section~\ref{iso_cof_section} where we prove Theorem~\ref{iso_cof_thm}. 

\begin{enumerate}
\item[$\bullet$] In Section~\ref{notation_section} we fix some notation. We also give a table that plays the role of a dictionary between our notation and that of Weiss-Pryor. The purpose of that table is to help the exposition of certain proofs, especially in Subsections~\ref{cof_fsp_subsection}, \ref{sos_good_subsection}.  

\item[$\bullet$] Section~\ref{holim_simpl_section} deals with homotopy limits in simplicial model categories. We follow Hirschhorn's style \cite[Chapters 18-19]{hir03}.   Since the homotopy limit is so ubiquitous in this work, we first give its definition in Subsection~\ref{holim_subsection}. Next, in the same subsection, we recall some of its  basic properties including the homotopy invariance (see Theorem~\ref{fib_cofib_thm}), the cofinality theorem (see Theorem~\ref{htpy_cofinal_thm}), and the Fubini theorem (see Theorem~\ref{fubini_thm}).  All these properties are indeed used in many places in this work. Subsection~\ref{cr_subsection} deals with cosimplicial replacement of a diagram. We prove Proposition~\ref{comm_prop}, which is the main new result of the section. It says that the canonical isomorphism $\underset{\mcald}{\text{holim}} \; F \cong \tot \; \pid F$ between the homotopy limit of a diagram $F \colon \mcald \lra \mcalm$ and the totalization of its cosimplicial replacement is natural in the following sense. If $\theta \colon \mcalc \lra \mcald$ is a functor between small categories,  then the obvious square involving the isomorphisms $\underset{\mcald}{\text{holim}} \; F \cong \tot \; \pid F$ and $\underset{\mcalc}{\text{holim}} \; F \theta \cong \tot \; \pid (F \theta)$ must commutes.  Proposition~\ref{comm_prop} will be used in the proof of Theorem~\ref{sos_thm}. 

\item[$\bullet$] Section~\ref{sos_good_section} proves two important results: Theorem~\ref{sos_thm} and Theorem~\ref{good_thm}. The first, which is the crucial ingredient in the proof of Theorem~\ref{main_thm},  roughly says that the homotopy limit  $\fsb (U) := \underset{V \in \bk(U)}{\text{holim}} \; F(V)$ does not depend on the choice of the basis $\mcalb$. Specifically, it says that  for any good basis $\mcalb$ for the topology of $M$, for any isotopy cofunctor $F \colon \okm \lra \mcalm$, the canonical map $\underset{V \in \ok(U)}{\text{holim}} \; F(V) \lra \underset{V \in \bk(U)}{\text{holim}} \; F(V)$ is a weak equivalence for all $U \in \om$. The proof of Theorem~\ref{sos_thm} goes through two big steps. The first step (see Subsection~\ref{cof_fsp_subsection}) consists of splitting $\fsb$ into smaller pieces $\fsbd  = \{\fsbp\}_{p \geq 0}$, and show that $\fsbp$ is independent of the choice of the basis $\mcalb$  for all $p$ (see Proposition~\ref{sosp_prop}). This idea of splitting comes from the paper of Weiss \cite{wei99}, and the nice thing is that  the collection $\fsbd$ turns out to be a cosimplicial object in the category of cofunctors from $\om$ to $\mcalm$. The second step, inspired by Pryor's work \cite{pryor15}, is to connect $\underset{[p] \in \Delta}{\text{holim}} \; \fsbp (U)$ and $\fsb(U)$ by a zigzag of natural weak equivalences (see Subsection~\ref{sos_good_subsection}). It is very important that every map of that zigzag is natural in both variables $U$ and $\mcalb$. This is  one of the reasons  we really need Section~\ref{holim_simpl_section} where all those maps are carefully inspected, especially the map that appears in Theorem~\ref{holim_tot_thm}.  Regarding Theorem~\ref{good_thm}, it  says that $\fsb$ is good provided that $F \colon \bkm \lra \mcalm$ is an isotopy cofunctor. This result is  a part of the proof of Theorem~\ref{main_thm}, and its proof is based on Theorem~\ref{sos_thm} and the Grothendieck construction (see Subsection~\ref{gro_const_subsection}). 

\item[$\bullet$] In Section~\ref{poly_section} we prove the main result of the first part: Theorem~\ref{main_thm}. To do this we use Theorem~\ref{sos_thm} and Theorem~\ref{good_thm} as mentioned earlier. We also use Lemmas~\ref{cofinal_lem}, \ref{poly_lem}, \ref{charac_lem}. The first lemma says that a certain functor is right cofinal. The second (which is a generalization of \cite[Theorem 4.1]{wei99})  states that if $F \colon \bkm \lra \mcalm$ is an isotopy cofunctor, then $\fsb$ is polynomial of degree $\leq k$. The proof of this result also uses the Grothendieck construction. The third lemma (which is a generalization of \cite[Theorem 5.1]{wei99}) is a characterization of polynomial cofunctors. Note that Lemma~\ref{poly_lem} and Lemma~\ref{charac_lem} are important themselves.

\item[$\bullet$] Section~\ref{hc_section} deals with homogeneous cofunctors, and is devoted to the proof of Theorem~\ref{main2_thm}. The key ingredient we need is Lemma~\ref{homo_lem}, which roughly says  that homogeneous cofunctors $\om \lra \mcalm$ of degree $k$ are determined by their values on subsets diffeomorphic to the disjoint union of exactly $k$ open balls provided that $\mcalm$ has a zero object. So Lemma~\ref{homo_lem} is also a useful result in its own right since it characterizes homogeneous cofunctors.   
Note that the proof of Lemma~\ref{homo_lem} is based on the results we obtained in Section~\ref{sos_good_section} and Section~\ref{poly_section}. 

\item[$\bullet$] Section~\ref{iso_cof_section} proves Theorem~\ref{iso_cof_thm}. To do this we use a completely different method (but rather lengthy) from that we used in previous sections. As mentioned earlier, the key concept here is that of \textit{admissible family} (see Definition~\ref{pw_adm_defn}) introduced in \cite{paul_don17}.  In Subsection~\ref{holim_subsectiong} we recall some useful properties for homotopy limits in general model categories. Subsections~\ref{dv_subsection}, \ref{hp_subsection} are preparatory subsections dealing with technical tools needed for the proof of Theorem~\ref{iso_cof_thm}. Lastly, Subsection~\ref{iso_cof_subsection} proves Theorem~\ref{iso_cof_thm}.     
\end{enumerate}

\textbf{Acknowledgements.} This work has been supported by  Pacific Institute for the Mathematical Sciences (PIMS) and the University of Regina, that the authors acknowledge. We are also grateful to P. Hirschhorn, J. Scherer, and W. Chacholski for helpful conversations (by emails) about homotopy limits and homotopy colimits.

\section{Notation}   \label{notation_section}

In this section we fix some notation. 

\begin{enumerate}
\item[$\bullet$] We let $M$ denote a smooth manifold. If $U$ is a subset of $M$, we let $\mcalo(U)$ denote the poset of open subsets of $U$, morphisms being inclusions of course. In particular one has the poset $\om$. 
\item[$\bullet$] For $k \geq 0$, and $U \in \om$, we let $\ok(U) \subseteq \mcalo(U)$ denote the full subposet whose objects are open subsets  diffeomorphic to the disjoint union of at most $k$ balls.  In particular one has the poset $\okm$. 
\item[$\bullet$] We write $\mcalo$ for the collection of all subsets of $M$ diffeomorphic to an open ball. Certainly $\mcalo$ is a full subposet of $\om$.
\item[$\bullet$] We let $\mcalb$ denote a good basis (see Definition~\ref{gb_defn}) for the topology of $M$. Clearly, one has $\mcalb \subseteq \mcalo$ for any good basis $\mcalb$. 
\item[$\bullet$] We write $\mcalm$ for a simplicial model category unless stated otherwise.
\item[$\bullet$] If $\beta \colon F \lra G$ is a natural transformation, the component of $\beta$ at $x$ will be denoted $\beta[x] \colon F(x) \lra G(x)$.  
\item[$\bullet$] We use the notation $x := \text{def}$ to state that the left hand side is defined by the right
hand side. 
\end{enumerate}

Since the proofs of some important results in this paper are based on \cite{pryor15} and \cite{wei99}, we need a dictionary of notations which is provided by the following table. The purpose of that table is then to help the exposition of certain proofs, especially in Subsections~\ref{cof_fsp_subsection}, \ref{sos_good_subsection} as we said before. The first column gives the notation that we use in this paper, while the second and the third regard the notation used in \cite{pryor15} and \cite{wei99} respectively. The notations that appear in the same row have the same meaning. The word \lq\lq nothing\rq\rq{} means that there is no notation with the same meaning in the corresponding paper. For instance, in the first row we have the notation $\mcalo$ in this paper, which stands for the maximal good basis for the topology of $M$. However there is no notation in \cite{pryor15} and \cite{wei99} that has the same meaning as $\mcalo$. 

\begin{center}
\begin{tabular}{c|c|c}
In this paper  & In Pryor's paper \cite{pryor15} & In Weiss' paper \cite{wei99} \\ \hline
$\mcalo$      &  nothing  & nothing \\ 
$\om$        &  $\om$ or just $\mcalo$ & $\om$ or just $\mcalo$ \\ 
$\okm$       & $\ok$  & $\mcalo k$ \\ 
$\bkm$ (see Definition~\ref{fsb_defn})    & $\bk$  & nothing \\ 
$\okp (M)$ (see Definition~\ref{bkp_defn}) & nothing  & $\mcali k \mcalo k_p (M)$ \\ 
$\bkp (M)$ (see Definition~\ref{bkp_defn}) & $\mcala_k (\bk)_p (M)$  & nothing \\
$\bkq (M)$  & $(\mcala_k)_q \bk (M)$  & nothing   \\ 
$\fso$ (see Example~\ref{fso_expl}) & nothing  & $E^{!}$  \\ 
$\fsb$ (see Definition~\ref{fsb_defn}) & $F^{!}$  & nothing  \\
$\ftb^p$ (see Definition~\ref{fsbp_defn})  &  $F_p$  & nothing   \\
$\ftb^{!p}$ (see Definition~\ref{fsbp_defn})  &  $F^{!}_p$ & nothing \\ 
$\fto^{!p}$ (see Definition~\ref{fsbp_defn})  & nothing & $E^{!}_p$ \\
$\fhb^q$ (see  (\ref{fhq}))  & $\widehat{F}_q$  &  nothing   \\
$\fhb^{!q}$ (see (\ref{fhsq}))  &  $\widehat{F}^{!}_q$ & nothing  
\end{tabular}
\end{center}

For the meaning of $\bkq (M)$ we refer the reader to the beginning of Subsection~\ref{sos_good_subsection}.

\section{Homotopy limits in simplicial model categories}  \label{holim_simpl_section}

In this section we recall some useful definitions and results about homotopy limits in simplicial model categories. We also prove Corollary~\ref{hir_coro} and Proposition~\ref{comm_prop}, which will be used in Section~\ref{sos_good_section}. The main reference here is Hirschhorn's book \cite[Chapter 18]{hir03}.

Let us begin with the following remark and notation. 

\begin{rmk}
For the sake of simplicity, all the functors $F \colon \mcalc \lra \mcalm$ in this section are covariant unless stated otherwise. However, in next sections our functors will be contravariant since manifold calculus deals with contravariant functors. This is not an issue of course since all the statements of this section hold for contravariant functors as well: it suffices to replace everywhere \lq\lq$\mcalc$\rq\rq{} by its opposite category \lq\lq$\mcalco$\rq\rq{}. 
\end{rmk}

\begin{nota}
The following standard notations will be used only in this section. 
\begin{enumerate}
\item[(i)] We let $\Delta$ denote the category whose objects are $[n] = \{0, \cdots, n\}, n \geq 0$, and whose morphisms are non-decreasing maps.  For $n \geq 0$, we let $\Delta[n]$ denote the simplicial set defined as $(\Delta[n])_p = \underset{\Delta}{\text{hom}} ([p], [n])$. 
\item[(ii)] If $\mcalc$ is a category, we write $N(\mcalc)$ for the nerve of $\mcalc$. If $c \in \mcalc$, we let $\mcalc \downarrow c$ denote the over category. An object of $\mcalc \downarrow c$ consists of a pair $(x, f)$, where $x \in \mcalc$ and $f \colon x \lra c$ is a morphism of $\mcalc$. A morphism from $(x, f)$ to $(x', f')$ consists of a morphism $g \colon x \lra x'$ of $\mcalc$ such that the obvious triangle commutes. 
\end{enumerate}
\end{nota}

\subsection{Homotopy limits}  \label{holim_subsection}

Here we  recall the definition of  the homotopy limit of a diagram in a simplicial model category. Next we recall some useful results due to  P. Hirschhorn \cite{hir03}. 

%\begin{defn} \label{cof_defn}
%Let $\mcalc$ and $\mcalm$ be categories. A \emph{cofunctor} from $\mcalc$ to $\mcalm$ is just a contravariant functor from $\mcalc$ to $\mcalm$. 
%\end{defn}

% Reference for the defn of holim: \cite[Definition 18.1.8]{hir03} \label{holim_defn}

\begin{defn} \label{holim_defn}
Let $\mcalm$ be a simplicial model category, and let $\mcalc$ be a small category. Consider a covariant functor $F \colon \mcalc \lra \mcalm$. The \emph{homotopy limit} of $F$, denoted $\underset{\mcalc}{\text{holim}} \; F$, is the object of $\mcalm$ defined to be the equalizer of the maps
\[
\xymatrix{\underset{c \in \mcalc}{\prod} (F(c))^{N(\mcalc \downarrow c)} \ar@<1ex>[r]^-{\phi} \ar@<-1ex>[r]_-{\psi}  & \underset{(f \colon c \ra c') \in \mcalc}{\prod} (F(c'))^{N(\mcalc \downarrow c)}. }
\]
\footnote{One of the axioms of the definition of a \textit{simplicial model category} $\mcalm$ \cite[Definition 9.1.6]{hir03} says that for an object $Y \in \mcalm$ and every simplicial set $K$, there exists an object $Y^K$ of $\mcalm$} Here $\phi$ and $\psi$ are defined as follows. Let $f \colon c \lra c'$ be a morphism of $\mcalc$. 
\begin{enumerate}
\item[$\bullet$] The projection of $\phi$ on the factor indexed by $f$ is the following composition where the first map is a projection
\[
\xymatrix{\underset{c \in \mcalc}{\prod} (F(c))^{N(\mcalc \downarrow c)} \ar[r] &  (F(c))^{N(\mcalc \downarrow c)} \ar[rr]^-{(F(f))^{N(\mcalc \downarrow c)}}  &   &  (F(c'))^{N(\mcalc \downarrow c)}.}
\]
\item[$\bullet$] The projection of $\psi$ on the factor indexed by $f$ is the following composition where the first map is again a projection. 
\[
\xymatrix{\underset{c \in \mcalc}{\prod} (F(c))^{N(\mcalc \downarrow c)} \ar[r] &  (F(c'))^{N(\mcalc \downarrow c')} \ar[rr]^-{(F(c'))^{N(\mcalc \downarrow f)}}  &   &  (F(c'))^{N(\mcalc \downarrow c)}.}
\]
\end{enumerate}
\end{defn}
 
\begin{defn}  Let $\mcalm$ be a category. 
Let $\mcalc$ and $\mcald$ be small categories, and let $\theta \colon \mcalc \lra \mcald$ be a functor. If $F \colon \mcald \lra \mcalm$ is an $\mcald$-diagram in $\mcalm$, then the composition $F \circ \theta \colon \mcalc \lra \mcalm$ is called the \emph{$\mcalc$-diagram in $\mcalm$ induced by $F$}, and it is denoted $\theta^* F$.   That is, 
\begin{eqnarray} \label{theta_starx}
\theta^* F := F \circ \theta. 
\end{eqnarray}
\end{defn}

The following proposition will be used in many places in this paper. Especially, we will use it to define morphisms between homotopy limits of diagrams of different shape. Also it will be used to show that certain diagrams commute. That proposition regards the change of the indexing category of a homotopy limit. 

\begin{prop}  \label{induced_holim_prop}
Let $\mcalm$ be a simplicial model category, and  let $\theta \colon \mcalc \lra \mcald$ be a functor between two small categories.  If $F \colon \mcald \lra \mcalm$ is an $\mcald$-diagram, then there is a canonical map 
\begin{eqnarray} \label{thetax}
[\theta; F] \colon \underset{\mcald}{\text{holim}}\; F \lra \underset{\mcalc}{\text{holim}}\; \theta^*F.
\end{eqnarray}
Furthermore, this map is natural in both variables $\theta$ and $F$. The naturality in $\theta$ says that if $\beta \colon \theta \lra \theta'$ is a natural transformation, then the following square commutes. 
\begin{eqnarray} \label{nat_theta}
\xymatrix{\underset{\mcald}{\text{holim}}\; F  \ar[rr]^-{[\theta; F]} &  &  \underset{\mcalc}{\text{holim}}\; \theta^*F  \\
\underset{\mcald}{\text{holim}}\; F  \ar[rr]_-{[\theta'; F]} \ar[u]^-{id} &  & \underset{\mcalc}{\text{holim}}\; \theta'^{*}F \ar[u]_-{\text{holim}(F\beta)} }
\end{eqnarray}
Here we have assumed $F$ contravariant (in the covariant case, one has to reverse the righthand vertical map). Regarding the naturality in $F$, it says that if $\eta \colon F \lra F'$ is a natural transformation, then the following square commutes. 
\begin{eqnarray} \label{theta_square}
\xymatrix{ \underset{\mcald}{\text{holim}}\; F \ar[rr]^-{[\theta; F]} \ar[d]_-{\text{holim}(\eta)}  &  &  \underset{\mcalc}{\text{holim}}\; \theta^*F \ar[d]^-{\text{holim}(\theta^*\eta)}  \\
 \underset{\mcald}{\text{holim}}\; F' \ar[rr]_-{[\theta; F']} &  &  \underset{\mcalc}{\text{holim}}\; \theta^*F'.}
\end{eqnarray}
\end{prop}

\begin{proof}
The construction of $[\theta; F]$ comes from the following observation, which provides a nice way to define a map between two equalizers. This observation will be also used in Subsection~\ref{cr_subsection}.  

 Consider the following diagrams in $\mcalm$. 
\[
\xymatrix{A \ar@<1ex>[r]^-{\alpha} \ar@<-1ex>[r]_-{\beta} & B}  \qquad \xymatrix{A' \ar@<1ex>[r]^-{\alpha'} \ar@<-1ex>[r]_-{\beta'} & B'}.  
\]
If $\Psi \colon A \lra A'$ is a map that satisfies the property
\begin{eqnarray} \label{eq_cond}
(\text{for all $g \colon E \lra A$}) \left((\alpha g = \beta g) \Rightarrow (\alpha'\Psi g = \beta' \Psi g)\right),
\end{eqnarray}
then we have an induced map 
\[
\widetilde{\Psi} \colon \text{eq}\left(\xymatrix{A \ar@<1ex>[r]^-{\alpha} \ar@<-1ex>[r]_-{\beta} & B}\right)  \lra  \text{eq}\left(\xymatrix{A' \ar@<1ex>[r]^-{\alpha'} \ar@<-1ex>[r]_-{\beta'} & B'} \right). 
\]
Now we define $[\theta; F] \colon \underset{\mcald}{\text{holim}}\; F  \lra \underset{\mcalc}{\text{holim}}\; \theta^*F$. By Definition~\ref{holim_defn}, the homotopy limit of $\theta^* X$ is the equalizer of the maps 
\[
\xymatrix{\underset{c \in \mcalc}{\prod} (F(\theta(c)))^{N(\mcalc \downarrow c)} \ar@<1ex>[r] \ar@<-1ex>[r]  & \underset{(c \lra c') \in \mcalc}{\prod} (F(\theta(c')))^{N(\mcalc \downarrow c)}. }
\]
Define 
\[
\Psi \colon \underset{d \in \mcald}{\prod} (F(d))^{N(\mcald \downarrow d)} \lra \underset{c \in \mcalc}{\prod} (F(\theta(c)))^{N(\mcalc \downarrow c)}
\]
as follows. For $c \in \mcalc$ the map from $\underset{d \in \mcald}{\prod} (F(d))^{N(\mcald \downarrow d)}$ to the factor indexed by $c$ is defined to be the composition
\begin{eqnarray} \label{fd_fc}
\xymatrix{\underset{d \in \mcald}{\prod} (F(d))^{N(\mcald \downarrow d)} \ar[r] & (F(\theta(c)))^{N(\mcald \downarrow \theta(c))} \ar[r]  & (F(\theta(c)))^{N(\mcalc \downarrow c)}},
\end{eqnarray}
where the first map is the projection onto the factor indexed by $\theta(c)$, and the second one is induced by the canonical functor $\mcalc \downarrow c \lra \mcald \downarrow \theta(c)$. It is straightforward to see that $\Psi$ satisfies condition (\ref{eq_cond}). We thus obtain  $[\theta; F] := \psit$. 

It is also straightforward to check that the squares (\ref{nat_theta}) and (\ref{theta_square}) commute. 
\end{proof}

%One can ask the question to know conditions on $\theta \colon \mcalc \lra \mcald$ to ensure that the natural map $[\theta; F]$ is a weak equivalence. This question is well known, and its answer is given by Theorem~\ref{htpy_cofinal_thm} below. Before we state it, we need to recall the following two definitions. 

We end this subsection with the following three important properties of homotopy limits. The first (see Theorem~\ref{fib_cofib_thm}) is known as the homotopy invariance for homotopy limits. The second (see Theorem~\ref{htpy_cofinal_thm}) is the cofinality theorem. And the last (see Theorem~\ref{fubini_thm}) is  the so-called Fubini Theorem for homotopy limits. Before we state those properties, we need the recall the following three definitions. 

\begin{defn} \label{uc_defn}
Let $\theta \colon \mcalc \lra \mcald$ be a functor, and let $d \in \mcald$. The \emph{under category} $d \downarrow \theta$ is defined as follows.  An object of $d \downarrow \theta$ is a pair $(c, f)$ where $c$ is an object of $\mcalc$ and $f$ is a morphism of $\mcald$ from $d$ to $\theta(c)$ . A morphism from $(c, f)$ to $(c', f')$ consists of a morphism $g \colon c \lra c'$ of $\mcalc$ such that the obvious triangle commutes (that is, $\theta(g) f = f'$). In similar fashion, one has the \textit{over category} $\theta \downarrow d$. 
 \end{defn}

\begin{defn} \cite[Definition 19.6.1]{hir03} \label{cofinal_defn}
A functor $\theta \colon \mcalc \lra \mcald$ is \emph{homotopy right cofinal} (respectively \emph{homotopy left cofinal}) if for every $d \in \mcald$, the under category  $d \downarrow \theta$ (respectively the over category $\theta \dar d$) (see Definition~\ref{uc_defn}) is contractible. 
\end{defn}

\begin{defn} \label{owf_defn}
If $\mcalc$ is a category and $\mcalm$ is a model category, a functor $F \colon \mcalc \lra \mcalm$ is said to be an \emph{objectwise fibrant} functor if the image of every object under $F$ is fibrant. 
\end{defn}

\begin{thm} \cite[Theorems 18.5.2, 18.5.3]{hir03}  \label{fib_cofib_thm}
Let $\mcalm$ be a simplicial model category, and let $\mcalc$ be a small category. 
\begin{enumerate}
\item[(i)] If a functor $F \colon \mcalc \lra \mcalm$  is objectwise fibrant (see Definition~\ref{owf_defn}), then the homotopy limit $\underset{\mcalc}{\text{holim}} \; F$ is a fibrant object of $\mcalm$. 
\item[(ii)] Let $\eta \colon F \lra G$ be a map of $\mcalc$-diagrams in $\mcalm$. Assume that both $F$ and $G$ are objectwise fibrant. If for every object $c$ of $\mcalc$ the component $\eta[c] \colon F(c) \lra G(c)$ is a weak equivalence, then the induced map of homotopy limits $\underset{\mcalc}{\text{holim}} \; F \lra \underset{\mcalc}{\text{holim}} \; G$ is a weak equivalence of $\mcalm$.
\end{enumerate} 
\end{thm}

\begin{thm}\cite[Theorem 19.6.7]{hir03}  \label{htpy_cofinal_thm}
 Let $\mcalm$ be a simplicial model category.  If $\theta \colon \mcalc \lra \mcald$ is homotopy right cofinal (respectively homotopy left cofinal), then for every objectwise fibrant contravariant (respectively covariant) functor $F \colon \mcald \lra \mcalm$, the natural map $[\theta; F]$ from Proposition~\ref{induced_holim_prop} is a weak equivalence. 
\end{thm}

\begin{thm} \label{fubini_thm}
Let $\mcalm$ be simplicial model category, and let $\mcalc$ and $\mcald$ be small categories. Let $F \colon \mcalc \times \mcald \lra \mcalm$ be a bifunctor. Then there exists a natural weak equivalence
\[
\underset{\mcalc}{\text{holim}} \; \underset{\mcald}{\text{holim}}\; F \stackrel{\sim}{\lra} \underset{\mcald}{\text{holim}} \; \underset{\mcalc}{\text{holim}}\; F. 
\]
\end{thm}

\begin{proof}
This is the dual of \cite[Theorem 24.9]{cha_sch01}. 
\end{proof}

\subsection{Cosimplicial replacement of a diagram}  \label{cr_subsection}

The goal of this subsection is to prove Corollary~\ref{hir_coro} and Proposition~\ref{comm_prop}. As we said before those two results will be used in Section~\ref{sos_good_section}.

\begin{defn} \label{cr_defn}
Let $\mcalm$ be a simplicial model category, and let $\mcalc$ be a small category. For a covariant functor $F \colon \mcalc \lra \mcalm$, we define its \emph{cosimplicial replacement}, denoted $\pid F \colon \Delta \lra \mcalm$,  as 
\[
\Pi^n F := \prod_{(c_0 \ra \cdots \ra c_n) \in N_n(\mcalc)} F ({c_n}). 
\]
For $0 \leq j \leq n-1$ the codegeneracy map $s^j \colon \Pi^n F \lra \Pi^{n-1} F$ is defined as follows. The projection of $s^j$  onto the factor indexed by $c_0 \ra \cdots \ra c_{n-1}$ is defined to be the projection of $\Pi^n F$ onto the factor indexed by $c_0 \ra \cdots \ra c_j \stackrel{id}{\lra} c_j \ra \cdots \ra c_{n-1}$. Cofaces are defined in a similar way.  
\end{defn}

\begin{rmk} \label{betax_rmk}
Let $\mcalm$ be a model category, and let $\mcalc$ and $\mcald$ be small categories. Consider a covariant functor $F \colon \mcald \lra \mcalm$. Also consider a functor $\theta \colon \mcalc \lra \mcald$. Recall the notation $\theta^*(-)$ from (\ref{theta_starx}). Then there exists a canonical map 
\[
\beta^{\bullet}_F \colon \Pi^{\bullet} F \lra \Pi^{\bullet} (\theta^* F) 
\] 
defined as follows.  The map from $\Pi^n F$ to the factor of $\Pi^n (\theta^* F)$ indexed by $c_0 \ra \cdots \ra c_n$ is just the projection of $\Pi^n F$ onto the factor $F(\theta(c_n))$ indexed by $\theta(c_0) \ra \cdots \ra \theta(c_n)$. 
\end{rmk}

The following proposition is stated (without any proof) in \cite[Chapter XI, Section 5]{bous_kan72} for diagrams of simplicial sets.

\begin{prop} \label{rf_thm}
Let $\mcalm$ be a simplicial model category. Let $\mcalc$ be a small category, and let $F \colon \mcalc \lra \mcalm$ be an objectwise fibrant covariant functor. Then the cosimplicial replacement $\Pi^{\bullet} F$ is \emph{Reedy fibrant} (see the definition of \lq\lq Reedy fibrant\rq\rq{}  in the proof). 
\end{prop}

\begin{proof}
First we recall the definition of \textit{Reedy fibrant}. Let $\zd \colon \Delta \lra \mcalm$ be a cosimplicial object in $\mcalm$. For $n \geq 0$, we let $\mcale_n$ denote the category whose objects are  maps $[n] \lra [p]$ of $\Delta$ such that $p < n$. A morphism from $[n] \lra [p]$ to $[n] \lra [q]$ consists of a map $[p] \lra [q]$ of $\Delta$ such that the obvious triangle commutes. The \textit{matching object} of $\zd$ at $[n] \in \Delta$, denoted $M^n \zd$, is defined to be the limit of the $\mcale_n$-diagram that sends $[n] \lra [p]$ to $Z^p$. That is,
\[
M^n \zd  := \underset{([n] \ra [p]) \in \mcale_n}{\text{lim}} \; Z^p.
\]
By the universal property, there exists a unique map $\alpha^n \colon Z^n \lra M^n \zd$ that makes certain triangles commutative. That map is induced by all codegeneracies $s^j \colon Z^n \lra Z^{n-1}, 0 \leq j \leq n-1$. We say that $\zd$ is \textit{Reedy fibrant} if $\alpha^n$ is a fibration for all $n \geq 0$. 

We come back to the proof of the proposition.  Let $n \geq 0$. Since each codegeneracy map $s^j \colon \Pi^n F \lra \Pi^{n-1} F$ is a projection (see Definition~\ref{cr_defn}), it follows that $\alpha^n$ is also a projection. This implies (by the assumption that $F(c)$ is fibrant for any $c \in \mcalc$) that $\alpha^n$ is a fibration, which completes the proof.  
\end{proof}

\begin{defn} \label{tot_defn}
Let $\mcalm$ be a simplicial model category. Let $Z^{\bullet} \colon \Delta \lra \mcalm$ be a cosimplicial object in $\mcalm$. The \emph{totalization} of $Z^{\bullet}$, denoted $\text{Tot} \; Z^{\bullet}$, is defined to be the equalizer of the maps
\[
\xymatrix{\underset{[n] \in \Delta}{\prod} (Z^n)^{\Delta[n]} \ar@<1ex>[r]^-{\phi'} \ar@<-1ex>[r]_-{\psi'}  & \underset{(f \colon [n] \ra [p]) \in \Delta}{\prod} (Z^p)^{\Delta[n]}. }
\]
Here the maps $\phi'$ and $\psi'$ are defined in the similar way as the maps $\phi$ and $\psi$ from Definition~\ref{holim_defn}. 
\end{defn}

\begin{thm} \cite[Theorem 12.5]{hir14} \label{holim_tot_thm}
Let $\mcalm$ be a simplicial model category. Let $\mcalc$ be a small category, and let $F \colon \mcalc \lra \mcalm$ be a covariant functor. Then there exists an isomorphism 
\begin{eqnarray} \label{phic_map}
\Phi_{\mcalc} \colon \underset{\mcalc}{\text{holim}} \; F \stackrel{\cong}{\lra} \text{Tot} \; \Pi^{\bullet} F,
\end{eqnarray}
which is natural in $F$.
\end{thm}

\begin{proof}
This is well detailed in \cite[Theorem 12.5]{hir14}. However, for our purposes, specifically for the proof of Proposition~\ref{comm_prop} below, we will  recall only the construction of  $\Phi_{\mcalc}$.  The map $\Phi_{\mcalc}$ is in fact  the composition of three isomorphisms (each obtained by using the observation we made at the beginning of the proof of Proposition~\ref{induced_holim_prop}): 
\[
\xymatrix{\underset{\mcalc}{\text{holim}} \; F \ar[rr]^-{\psit_{1\mcalc}}_-{\cong} &  & X \ar[rr]^-{\psit_{2\mcalc}}_-{\cong}  &  &  X' \ar[rr]^-{\psit_{3\mcalc}}_-{\cong} &  & \tot \; \pid F, }
\]
where 
\begin{enumerate}
\item[$\bullet$] $X$ is the equalizer of a diagram 
  \[
     \xymatrix{ \underset{c \in \mcalc, n \geq 0,  \Delta[n] \ra N(\mcalc \downarrow c)}{\prod} (F(c))^{\Delta[n]} \ar@<1ex>[rr]  \ar@<-1ex>[rr]  &  &  Y }
  \]
  \item[$\bullet$] $X'$ is the equalizer of a diagram 
  \[
     \xymatrix{ \underset{n \geq 0,  (c_0 \ra \cdots \ra c_n) \in N_n(\mcalc)}{\prod} (F(c_n))^{\Delta[n]} \ar@<1ex>[rr]  \ar@<-1ex>[rr]  &  &  Y' }. 
  \]
  \item[$\bullet$] By the definition of the cosimplicial replacement (see Definition~\ref{cr_defn}), and by the the definition of the totalization (see Definition~\ref{tot_defn}), one can easily see that $\tot \; \pid F$ is the equalizer of a diagram 
   \[
     \xymatrix{ \underset{n \geq 0,  (c_0 \ra \cdots \ra c_n) \in N_n(\mcalc)}{\prod} (F(c_n))^{\Delta[n]} \ar@<1ex>[rr]  \ar@<-1ex>[rr]  &  &  Y''}. 
  \]
  Since we are only interested in the definition of maps, it is not important here to know the definition of $Y$, $Y'$, and $Y''$.  
\end{enumerate}
Recalling  $\underset{\mcalc}{\text{holim}} \; F$ from Definition~\ref{holim_defn}, the map $\psit_{1\mcalc}$ is induced by the map 
\begin{eqnarray} 
\Psi_{1 \mcalc} \colon \underset{ c \in \mcalc}{\prod}  (F(c))^{N(\mcalc \downarrow c)}  \lra \underset{c \in \mcalc, n \geq 0,  \Delta[n] \ra N(\mcalc \downarrow c)}{\prod} (F(c))^{\Delta[n]} , 
\end{eqnarray}
which is defined as follows. The projection of $\Psi_{1 \mcalc}$ onto the factor indexed by $(c \in \mcalc, n \geq 0, \Delta[n] \stackrel{\sigma}{\lra} N(\mcalc \downarrow c))$ is the composition 
\[
\xymatrix{ \underset{ c \in \mcalc}{\prod}  (F(c))^{N(\mcalc \downarrow c)} \ar[r] & (F(c))^{N(\mcalc \downarrow c)} \ar[r]  &  (F(c))^{\Delta[n]},}
\]
where the first map is the projection  onto the factor indexed by $c$, and the second is the canonical map induced by $\sigma \colon \Delta[n] \lra N(\mcalc \downarrow c)$. 

Regarding the map $\psit_{2\mcalc}$, it is induced by the map 
\begin{eqnarray}
\Psi_{2 \mcalc} \colon \underset{c \in \mcalc, n \geq 0,  \Delta[n] \ra N(\mcalc \downarrow c)}{\prod} (F(c))^{\Delta[n]}  \lra \underset{n \geq 0,  (c_0 \ra \cdots \ra c_n) \in N_n(\mcalc)}{\prod} (F(c_n))^{\Delta[n]},
\end{eqnarray}
which is defined as follows. The projection of $\Psi_{2 \mcalc}$ onto the factor indexed by $(n \geq 0, c_0 \ra \cdots \ra c_n)$ is just the projection 
\[
 \underset{c \in \mcalc, n \geq 0,  \Delta[n] \ra N(\mcalc \downarrow c)}{\prod} (F(c))^{\Delta[n]} \lra (F(c_n))^{\Delta[n]}
\]
onto the factor indexed by $(c_n, c_0 \ra \cdots \ra c_n, c_n \stackrel{id}{\lra} c_n)$. 

Lastly, the map $\psit_{3 \mcalc}$ is induced by the identity map. So 
\begin{eqnarray} \label{psi3_eq}
\Psi_{3\mcalc} := id.
\end{eqnarray}  
\end{proof}

\begin{thm} \cite[Theorem 19.8.7]{hir03} \label{hir_thm}
Let $\mcalm$ be a simplicial model category, and let $Z^{\bullet} \colon \Delta \lra \mcalm$ be a cosimplicial object in $\mcalm$. If $Z^{\bullet}$ is Reedy fibrant then the  Bousfield-Kan map $\tot \; Z^{\bullet} \lra \underset{\Delta}{\text{holim}} \; Z^{\bullet}$ (see \cite[Definition 19.8.6]{hir03})  is a weak equivalence, which is natural in $Z^{\bullet}$. 
\end{thm}

We end this section with the following corollary and proposition. These results will be used in the course of the proof of Theorem~\ref{sos_thm} and Theorem~\ref{good_thm}, which will be done at the end of Subsection~\ref{sos_good_subsection}

\begin{coro} \label{hir_coro}
Let $\mcalm$ be a simplicial model category. Let $\mcalc$ be a small category, and let $F \colon \mcalc \lra \mcalm$ be an objectwise fibrant covariant functor. Then the Bousfield-Kan map $\text{Tot} \; \Pi^{\bullet} F \lra \underset{\Delta}{\text{holim}} \; \Pi^{\bullet} F$ (see \cite[Definition 19.8.6]{hir03}) is a weak equivalence, which is natural in $\pid X$.   
\end{coro}

\begin{proof}
This follows directly from Proposition~\ref{rf_thm} and Theorem~\ref{hir_thm}.  
\end{proof}

\begin{prop} \label{comm_prop}
Let $\mcalm$ be a simplicial model category, and let $\theta \colon \mcalc \lra \mcald$ be a functor between small categories. Consider a covariant functor $F \colon \mcald \lra \mcalm$. Also consider the maps $[\theta; F], \beta^{\bullet}_F$, and $\Phi_{\mcalc}$ from Proposition~\ref{induced_holim_prop}, Remark~\ref{betax_rmk}, and Theorem~\ref{holim_tot_thm} respectively. Then the following square commutes. 
\[
\xymatrix{ \underset{\mcald}{\text{holim}} \; F \ar[rr]^-{\Phi_{\mcald}}_-{\cong}  \ar[d]_-{[\theta; F]}  &  &  \text{Tot} \; \pid F \ar[d]^-{\text{Tot}\; \beta^{\bullet}_F}    \\
 \underset{\mcalc}{\text{holim}} \; \theta^* F  \ar[rr]_-{\Phi_{\mcalc}}^-{\cong}  &  &  \text{Tot} \; \pid (\theta^*F).}
\]
\end{prop}

Warning! Proposition~\ref{comm_prop} does not follow from the naturality  of the map $\Phi_{\mcalc}$ from Theorem~\ref{holim_tot_thm}. This is because $F$ and $\theta^* F$ does not have the same domain. So to prove Proposition~\ref{comm_prop} we really have to use the definition of $\Phi_{\mcalc}$. 

\begin{proof}[Proof of Proposition~\ref{comm_prop}]
Recall the maps $\Psi_{i (-)}, 1 \leq i \leq 3,$ from the proof of Theorem~\ref{holim_tot_thm}. To prove the proposition, it suffices to see  that the  three  squares induced by the pairs $(\Psi_{i \mcald}, \Psi_{i \mcalc}), 1 \leq i \leq 3,$ are all commutative. Let us begin with the following square induced by $(\Psi_{1 \mcald}, \Psi_{1 \mcalc})$. 
\begin{eqnarray} \label{sq1_eqn}
\xymatrix{ \underset{d \in \mcald}{\prod} (F(d))^{N(\mcald \dar d)} \ar[rr]^-{\Psi_{1\mcald}}  \ar[d]_-{\alpha}  &  &    \underset{d \in \mcald, n \geq 0,  \Delta[n] \ra N(\mcald \downarrow d)}{\prod} (F(d))^{\Delta[n]} \ar[d]^-{\lambda} \\
                \underset{c \in \mcalc}{\prod} (F(\theta(c)))^{N(\mcalc \dar c)}  \ar[rr]_-{\Psi_{1\mcalc}}  &   &   \underset{c \in \mcalc, n \geq 0,  \Delta[n] \ra N(\mcalc \downarrow c)}{\prod} (F(\theta(c)))^{\Delta[n]}. }
\end{eqnarray}
Here $\alpha$ is the composition from (\ref{fd_fc}), while the projection of $\lambda$ onto the factor indexed by $(c, n,  \sigma \colon \Delta[n] \ra N(\mcalc \dar c))$ is the projection  
\[
\underset{d \in \mcald, n \geq 0,  \Delta[n] \ra N(\mcald \downarrow d)}{\prod} (F(d))^{\Delta[n]}  \lra F(\theta(c))^{\Delta[n]}
\]
onto the factor indexed by $(\theta(c), n, \Delta[n] \stackrel{\sigma}{\lra} N(\mcalc \dar c) \stackrel{f}{\lra} N(\mcald \dar \theta(c)))$, where $f$ is induced by the obvious functor $\mcalc \dar c \lra \mcald \dar \theta(c)$.  Using the definitions, it is straightforward to  check that the square (\ref{sq1_eqn}) commutes.  

It is also straightforward to  see that the following square, induced by the pair $(\Psi_{2\mcald}, \Psi_{2\mcalc})$, is commutative. 
\[
\xymatrix{ \underset{d \in \mcald, n \geq 0,  \Delta[n] \ra N(\mcald \downarrow d)}{\prod} (F(d))^{\Delta[n]} \ar[d]_-{\lambda}   \ar[rr]^-{\Psi_{2\mcald}} &  &  \underset{n\geq 0, (d_0 \ra \cdots \ra d_n) \in N_n(\mcald)}{\prod}  (F(d_n))^{\Delta[n]}  \ar[d]    \\
 \underset{c \in \mcalc, n \geq 0,  \Delta[n] \ra N(\mcalc \downarrow d)}{\prod} (F(\theta(c)))^{\Delta[n]}  \ar[rr]_-{\Psi_{2 \mcalc}}  &  &   \underset{n\geq 0, (c_0 \ra \cdots \ra c_n) \in N_n(\mcalc)}{\prod} (F(\theta(c_n)))^{\Delta[n]} } 
\]
Lastly, the square induced by $(\Psi_{3\mcald}, \Psi_{3\mcalc})$ is clearly commutative since $\Psi_{3\mcald} = id$ and  $\Psi_{3\mcalc} = id$ by (\ref{psi3_eq}). 
\end{proof}

\section{Special open sets and good cofunctors}    \label{sos_good_section}

The goal of this section is to prove Theorem~\ref{sos_thm} and Theorem~\ref{good_thm} below. The first result is a key ingredient, which will be used in many places throughout Sections~\ref{sos_good_section}, \ref{poly_section}, \ref{hc_section}. It roughly says that a certain homotopy limit $\underset{V \in \bk(U)}{\text{holim}} \; F(V)$ is independent  of the choice of the basis $\mcalb$.  The second theorem is a part of the proof of the main result of this paper (Theorem~\ref{main_thm}). Note that the subsections~\ref{cof_fsp_subsection}, \ref{sos_good_subsection} are influenced by the work of Pryor \cite{pryor15}. 

From now on, we let $M$ denote a smooth manifold, and we let $\om$ to be the poset of open subsets of $M$. Also, for $k \geq 0$, we let $\okm \subseteq \om$ to be the full subposet whose objects are open subsets diffeomorphic to the disjoint union of at most $k$  balls. Below (see Example~\ref{fso_expl}) we will see that $\okm$ can be obtained in another way. In \cite{wei99} Weiss calls objects of $\okm$ \textit{special open sets}. 
%Let $\mcalm$ be a \textit{simplicial model category} (\cite[Definition 16.6.21]{hir03}).

\begin{thm} \label{sos_thm}
Let $\mcalm$ be a  simplicial model category, and let $\mcalb$ and $\mcalb'$ be good bases (see Definition~\ref{gb_defn}) for the topology of $M$ such that $\mcalb \subseteq \mcalb'$. Let $\mcalb'_k(M) \subseteq \okm$ denote the full subcategory whose objects are disjoint unions of at most $k$ elements from $\mcalb'$. Consider an isotopy cofunctor (see Definition~\ref{isotopy_cof_defn}) $F \colon \mcalb'_k(M) \lra \mcalm$. Also consider the cofunctors $\fsb, \fsbpr \colon \om \lra \mcalm$ from Definition~\ref{fsb_defn}. Then the natural map 
\[
[\theta; F] \colon \fsbpr \lra \fsb
\]
induced by the inclusion functor $\theta \colon \bk(M) \lra \mcalb'_k(M)$  is a weak equivalence. Here the notation \lq\lq[-; -] \rq\rq{} comes from Proposition~\ref{induced_holim_prop}. 
\end{thm}

\begin{thm} \label{good_thm}
Let $\mcalm$ be a  simplicial model category, and let $\mcalb$ be a good basis (see Definition~\ref{gb_defn}) for the topology of $M$. Consider the poset $\bkm$ from Definition~\ref{fsb_defn}, and let $F \colon \bkm \lra \mcalm$ be an isotopy cofunctor (see Definition~\ref{isotopy_cof_defn}). Then the cofunctor $\fsb \colon \om \lra \mcalm$  from  Definition~\ref{fsb_defn} is good (see Definition~\ref{good_defn}). 
\end{thm}

The proof of Theorem~\ref{sos_thm} and Theorem~\ref{good_thm} will be done in  Subsection~\ref{sos_good_subsection} after some preliminaries results.

\subsection{Isotopy cofunctors}

The goal of this subsection is to prove Proposition~\ref{fsrik_prop}, which will be used in Sections~\ref{poly_section}, \ref{hc_section}. Note that this result is well known in the context of topological spaces.

%Let $M$ be a smooth manifold, and let $\om$ be the poset of open subsets of $M$. Of course morphisms of $\om$ are inclusions.   

%m $f \colon U \hra U'$ of $\om$ is said to be an \emph{isotopy equivalence} if there exists a smooth embedding $g \colon U' \lra U$ such that $gf$ and $fg$ are both smoothly isotopic to $id_U$ and $id_{U'}$ respectively. 
%\[
%L \colon U \times [0, 1] \lra U', \quad (x, t) \mapsto L_t(x) := L(x, t)
%\]
%that satisfies the following three conditions:
%\begin{enumerate}\begin{defn}
%A morphis
%\item[(a)] $L_0 \colon U \hra U'$ is the inclusion map;
%\item[(b)] $L_1(U) = U'$;
%\item[(c)] for all $t$, $L_t \colon U \lra U'$ is a smooth embedding. 
%\end{enumerate}
%Such a map $L$ is called an \emph{isotopy from $U$ to $U'$}. 
%\end{defn}

We begin with several definitions. The first one is the notion of isotopy equivalence, which  is well known in differential topology, manifold calculus, and other areas. Nevertheless we  need to recall it  for our purposes in Section~\ref{iso_cof_section}.   

\begin{defn}  \label{iso_eq_defn}
A morphism $U \hra U'$ of $\om$ is said to be an \emph{isotopy equivalence} if there exists a continuous map
\[
L \colon U \times [0, 1] \lra U', \quad (x, t) \mapsto L_t(x) := L(x, t)
\]
that satisfies the following three conditions:
\begin{enumerate}
\item[(a)] $L_0 \colon U \hra U'$ is the inclusion map;
\item[(b)] $L_1(U) = U'$;
\item[(c)] for all $t$, $L_t \colon U \lra U'$ is a smooth embedding. 
\end{enumerate}
Such a map $L$ is called an \emph{isotopy from $U$ to $U'$}. 
\end{defn}

\begin{defn} \label{isotopy_cof_defn}
Let $\mcalc \subseteq \om$ be a subcategory of $\om$, and let $\mcalm$ be a  model category.  A  cofunctor $F \colon \mcalc \lra \mcalm$ is called \emph{isotopy cofunctor} if it satisfies the following two conditions:
\begin{enumerate}
\item[(a)] $F$ is objectwise fibrant (see Definition~\ref{owf_defn});
\item[(b)] $F$ sends isotopy equivalences (see Definition~\ref{iso_eq_defn}) to weak equivalences. 
\end{enumerate}
  
\end{defn}

\begin{defn}   \label{good_defn}
 Let $\mcalm$ be a simplicial model category. 
A cofunctor $F \colon \om \lra \mcalm$ is called \emph{good} if it satisfies the following two conditions:
\begin{enumerate}
\item[(a)] $F$ is an isotopy cofunctor (see Definition~\ref{isotopy_cof_defn});
\item[(b)] For any string $U_0 \ra U_1 \ra \cdots$ of inclusions of $\om$, the natural map 
\[
F\left(\bigcup_{i=0}^{\infty} U_i\right) \lra \underset{i}{\text{holim}} \; F(U_i)
\]
is a weak equivalence. 
\end{enumerate} 
\end{defn}

In order words, a cofunctor $F \colon \om \lra \mcalm$ is \textit{good} if it satisfies three conditions: (a), (b) from Definition~\ref{isotopy_cof_defn}, and (b) from Definition~\ref{good_defn}. As we said in the introduction, this definition is slightly different from the classical one \cite[Page 71]{wei99} (see the comment we made right after Theorem~\ref{main_thm}).

\begin{defn}  \label{gb_defn}
A basis for the topology of $M$ is called \emph{good} if each element in there is diffeomorphic to an open ball. 
\end{defn}

\begin{defn}  \label{fsb_defn} 
Let $\mcalb$ be a good basis (see Definition~\ref{gb_defn}) for the topology of $M$.  
\begin{enumerate}
\item[(i)] For $k \geq 0$, we define $\bkm \subseteq \om$ to be the full subposet whose objects are disjoint unions of at most $k$ elements from $\mcalb$. 
\item[(ii)] If $F \colon \bkm \lra \mcalm$ is a cofunctor, we  define $F^{!}_{\mcalb} \colon \om \lra \mcalm$ as 
\[
\fsb (U) := \underset{V \in \bk(U)}{\text{holim}} \; F(V).
\]
\end{enumerate}
\end{defn}

\begin{expl}  \label{fso_expl}
Let $\mcalo$ be the collection of all subsets of $M$ diffeomorphic to an open ball. Certainly this is a good basis (see Definition~\ref{gb_defn}) for the topology of $M$.  So, by Definition~\ref{fsb_defn}, one has the poset $\okm$, which is exactly the same as the poset $\okm$ we defined just before Theorem~\ref{sos_thm}. If $F \colon \okm \lra \mcalm$ is a cofunctor, one also has the cofunctor $\fso \colon \om \lra \mcalm$ defined as
\[
\fso (U) := \underset{V \in \ok(U)}{\text{holim}} \; F(V). 
\]
Clearly $\mcalo$ is the biggest (with respect to the inclusion) good basis for the topology of $M$. 
\end{expl}

As we said before, the following proposition will be used in Section~\ref{poly_section} and Section~\ref{hc_section}.  

\begin{prop} \label{fsrik_prop} 
Let $\mcalb$ and $\bkm$ as in Definition~\ref{fsb_defn}. Let $\mcalm$ be a simplicial model category, and 
let $F \colon \bkm \lra \mcalm$ be an objectwise fibrant cofunctor. 
\begin{enumerate}
\item[(i)] Then there is a natural transformation $\eta$ from $F$ to  the restriction $\fsb | \bkm$, which is an objectwise weak equivalence.  
\item[(ii)] If in addition $F$ is  an isotopy cofunctor (see Definition~\ref{isotopy_cof_defn}), then so is the restriction of $\fsb$ to $\bkm$. 
\end{enumerate} 
\end{prop}

\begin{proof}
\begin{enumerate}
\item[(i)] Let $U \in \bkm$, and let $\theta, \theta' \colon \bk(U) \lra \bk(U)$ be functors defined as  $\theta(V) = V$ and $\theta'(V) = U$. Certainly there is a natural transformation $\beta \colon \theta \lra \theta'$. This induces by (\ref{nat_theta}) the following commutative square. 
\[
\xymatrix{\underset{V \in \bk(U)}{\text{holim}}\; F(V)  \ar[rr]^-{[\theta; F]} &  &  \underset{V \in \bk(U)}{\text{holim}}\; F(V)  \\
\underset{V \in \bk(U)}{\text{holim}}\; F(V)  \ar[rr]_-{[\theta'; F]} \ar[u]^-{id} &  & \underset{V \in \bk(U)}{\text{holim}}\; F(U). \ar[u]_-{\text{holim}(F\beta)} }
\]
Clearly one has $F(U) \simeq \underset{V \in \bk(U)}{\text{holim}}\; F(U)$. This allows us to define $\eta[U] := \text{holim} (F \beta)$.   Since $\theta$ is the identity functor, it follows that the map $[\theta; F]$ is a weak equivalence (in fact it is the identity functor as well). The map $[\theta'; F]$ is also a weak equivalence (by Theorem~\ref{htpy_cofinal_thm}) since $\theta'$ is homotopy right cofinal. Indeed,  for every $V \in \bk(U)$ the under category (see Definition~\ref{uc_defn}) $V \downarrow \theta'$ has a terminal object, namely $(U, V \hra U)$.  Now, applying the two-out-of-three axiom we deduce that the map $\text{holim} (F\beta)$ is a weak equivalence. Regarding the naturality of $\eta[U]$ in $U$, it follows easily from  (\ref{theta_square}).
\item[(ii)] Certainly the functor $\fsb|\bk(M)$ satisfies condition (a) from Definition~\ref{isotopy_cof_defn} because of Theorem~\ref{fib_cofib_thm}.  Condition (b) from the same definition is also satisfied by $\fsb|\bk(M)$ (this follows directly from part (i)). 
\end{enumerate}
\end{proof}

\subsection{The cofunctors $F^{!p}$}  \label{cof_fsp_subsection}

In this subsection we consider a good basis $\mcalb$ and the poset $\bkm$ as in Definition~\ref{fsb_defn}. Also we consider the basis $\mcalo$ from Example~\ref{fso_expl}. The main results here are Proposition~\ref{sosp_prop} and Proposition~\ref{isop_prop} whose proofs are inspired by Pryor's work \cite{pryor15}. Those propositions are one of the key ingredients  in proving Theorem~\ref{sos_thm} and Theorem~\ref{good_thm}.

\begin{defn} \label{bkp_defn}
Let $k, p\geq 0$. 
\begin{enumerate}
\item[(i)] Define $\bkp(M)$ to be the poset whose objects are strings $V_0 \ra \cdots \ra V_p$ of $p$ composable morphisms in $\bkm$. A morphism from  $V_0 \ra \cdots \ra V_p$ to $W_0 \ra \cdots \ra W_p$ consists of a collection $\{f_i \colon V_i \hra W_i\}_{i=0}^p$ of isotopy equivalences such that all the obvious squares commute. 
\item[(ii)] Taking $\mcalb$ to be $\mcalo$, we have the poset $\okp(M)$. 
\end{enumerate}
\end{defn}

The following remark claims that the collection $\bkd(M) = \{\bkp(M)\}_{p \geq 0}$ is equipped with a canonical simplicial object  structure. 

\begin{rmk} \label{bkp_rmk}
Let $U \in \om$. For $0 \leq i \leq p+1$ define $d_i \colon \widetilde{B}_{k, p+1}(U) \lra \bkp(U)$ as 
\[
d_i(V_0 \ra \cdots \ra V_{p+1}) = \left\{ \begin{array}{ccc}
                                                               V_1 \ra \cdots \ra V_{p+1} & \text{if} & i =0  \\
                                                               V_0 \ra \cdots V_{i-1} \ra V_{i+1} \ra \cdots \ra V_{p+1}  & \text{if}  & 1 \leq i \leq p  \\
                                                               V_0 \ra \cdots \ra V_p  & \text{if}  & i = p+1.
                                                             \end{array} \right.
\]
For $0 \leq j \leq p$ define $s_j \colon \bkp(U) \lra \widetilde{B}_{k, p+1}(U)$ as 
\[
s_j(V_0 \ra \cdots \ra V_p)   =  V_0 \ra \cdots \ra V_j \stackrel{\text{id}}{\ra} V_j \ra \cdots \ra V_p.
\]
One can easily check that $d_i$ and $s_j$ satisfy the simplicial relation. So $\bkd(U)$ is a simplicial object in Cat, the category of small categories. 
\end{rmk}

\begin{defn} \label{fsbp_defn}
Let $F \colon \bkm \lra \mcalm$ be a cofunctor. 
\begin{enumerate}
\item[(i)] Define a cofunctor $\fsbp \colon \om \lra \mcalm$ as 
\[
\fsbp (U) := \underset{\bkp(U)}{\text{holim}} \; \ftpb,
\]
where 
\[\ftpb \colon \bkp(U) \lra \mcalm, \quad V_0 \ra \cdots \ra V_p \mapsto F(V_0).
\]
\item[(ii)] Taking again $\mcalb$ to be $\mcalo$, we have the cofunctor $\fsop \colon \om \lra \mcalm$. 
\end{enumerate} 
\end{defn}

The following remark will be used in Subsection~\ref{sos_good_subsection}. 

\begin{rmk} \label{fsp_rmk}
Let $U \in \om$, and let $\bkd(U)$ be the simplicial object from Remark~\ref{bkp_rmk}. Using the simplicial structure on  $\bkd(U)$, one can endow the collection $\fsbd (U) = \{\fsbp(U) \}_{p \geq 0}$ with a canonical cosimplicial structure as follows. First recall the notation $\theta^{*}(-)$ from (\ref{theta_starx}). Also recall the notation $[-;-]$ introduced in Proposition~\ref{induced_holim_prop}. Let $d_i$ and $s_j$ as in Remark~\ref{bkp_rmk}, and consider $d_0 \colon \widetilde{\mcalb}_{k, p+1}(U) \lra \bkp(U)$. Also consider the natural transformation $\beta \colon d_0^* \ftpb \lra \tilde{F}^{p+1}_{\mcalb}$ defined as 
\[
\beta[V_0 \stackrel{f}{\ra} V_1 \ra \cdots \ra V_{p+1}] := F(V_1)  \stackrel{F(f)}{\lra} F(V_0).
\]
Now define $d^i \colon \fsbp (U) \lra \tilde{F}^{!(p+1)}_{\mcalb} (U)$ as 
\[
d^i = \left\{ \begin{array}{ccc}
                   \text{holim}(\beta) \circ [d_0; \ftpb]  & \text{if}  & i =0  \\
                   \left[d_i; \ftpb\right]  & \text{if}  & 1 \leq i \leq p+1. 
                  \end{array} \right.
\]
Also define $s^j \colon \tilde{F}^{!(p+1)}_{\mcalb} (U) \lra \fsbp(U), 0 \leq j \leq p+1$, as 
\[
s^j = [s_j; \tilde{F}^{p+1}_{\mcalb}]. 
\]
Certainly the maps $d^i$ and $s^j$ satisfy the cosimplicial relations. So $\fsbd(U)$ is a cosimplicial object in $\mcalm$ for any $U \in \om$. 
\end{rmk}

\begin{defn} \label{loc_defn}
Let $\mcalm$ be a category with a class of weak equivalences, and let $\mcalc$ be any other category. A functor $\mcalc \lra \mcalm$ is called \emph{locally constant} if it sends every morphism of $\mcalc$ to a weak equivalence. 
\end{defn}

\begin{expl} \label{loc_expl}
Let $F \colon \bkm \lra \mcalm$ be an isotopy  cofunctor (see Definition~\ref{isotopy_cof_defn}). Then the cofunctor $\ftpb$ from Definition~\ref{fsbp_defn} is locally constant. This follows directly from the definition of a morphism of $\bkp(M)$ (see Definition~\ref{bkp_defn}) and the definition of an isotopy cofunctor. 
\end{expl}
 
 Now we state and prove the main results of this subsection (Proposition~\ref{sosp_prop} and Proposition~\ref{isop_prop}). First we need three preparatory lemmas.

\begin{lem}  \label{cisinski_lem}
Let $\mcalm$ be a simplicial model category. Let $\theta \colon \mcalc \lra \mcald$ be a functor between small categories. Consider a functor $G \colon \mcald \lra \mcalm$, and assume that it is locally constant (see Definition~\ref{loc_defn}). Also assume that the nerve of $\theta$ is a weak equivalence. Then the canonical map 
\[
[\theta; G] \colon \underset{\mcald}{\text{holim}} \; G \lra \underset{\mcalc}{\text{holim}} \; \theta^{*} G
\]
(see Proposition~\ref{induced_holim_prop}) is a weak equivalence. 
\end{lem}

\begin{proof}
This is just the dual of Proposition 1.17 from \cite{cis09}.
%\footnote{We have to be careful about the definition of holim or hocolim in \cite{cis09}: is that definition equivalent to ours?}. 
\end{proof}

\begin{lem} \label{pryor1_lem}
For any $U \in \om$ the nerve of the inclusion functor $\bkp(U) \hra \okp(U)$ is a homotopy equivalence for all $k, p \geq 0$.
\end{lem}

\begin{proof}
This is done in the course of the proof of Theorem 6.12 from \cite{pryor15}. 
\end{proof}

\begin{lem} \cite[Lemma 6.8]{pryor15}  \label{pryor2_lem}
Let $f \colon U \hra U'$ be a morphism of $\om$. If $f$ is an isotopy equivalence, then the nerve of the inclusion functor $\bkp(U) \hra \bkp(U')$ is a homotopy equivalence. 
\end{lem}

\begin{prop} \label{sosp_prop}
Let $\mcalm$ be a simplicial model category, and let $F \colon \okm \lra \mcalm$ be an isotopy cofunctor (see Definition~\ref{isotopy_cof_defn}).  Let $U$ be an object of $\om$.  Consider   $\theta \colon \bkp(U) \hra \okp(U)$,  the inclusion functor. Also consider $\ftpo \colon \okp(U) \lra \mcalm$, the cofunctor from Definition~\ref{fsbp_defn}.   Then  the canonical map  $[\theta; \ftpo] \colon \fsop(U) \lra \fsbp(U)$ is a weak equivalence for all $p \geq 0$.  Furthermore that map is natural in $U$. 
\end{prop}

\begin{proof}
Set $\mcalc := \bkp(M), \mcald := \okp(M),$ and $G := \ftpo$. Since $G$  is locally constant by Example~\ref{loc_expl}, and since the nerve of $\theta$ is a weak equivalence by Lemma~\ref{pryor1_lem}, it follows that $[\theta; \ftpo]$ is a weak equivalence by Lemma~\ref{cisinski_lem}. The naturality in $U$ comes directly from the definitions.   
\end{proof}

\begin{prop} \label{isop_prop}
Let $\mcalm$ be a simplicial model category, and let $F \colon \bkm \lra \mcalm$ be an isotopy cofunctor (see Definition~\ref{isotopy_cof_defn}). Then for any $p \geq 0$ the cofunctor $\fsbp$ (see Definition~\ref{fsbp_defn}) is an isotopy cofunctor as well. 
\end{prop} 

\begin{proof}
Let $U \hra U'$ be an isotopy equivalence, and let $\theta \colon \bkp(U) \hra \bkp(U')$ denote the inclusion functor. Consider the cofunctor $\ftpb \colon \bkp(U') \lra \mcalm$ from Definition~\ref{fsbp_defn}. Since $\ftpb$ is locally constant by Example~\ref{loc_expl}, and since the nerve of $\theta$ is a weak equivalence by Lemma~\ref{pryor2_lem}, the desired result follows by Lemma~\ref{cisinski_lem}.  
\end{proof}

\subsection{Grothendieck construction}  \label{gro_const_subsection}

In this subsection we recall the Grothendieck construction, and we give some examples that will be used further. We also recall an important result (see Theorem~\ref{cha_sch_thm}), which regards the homotopy limit of a diagram indexed by the Grothendieck construction.

\begin{defn} \label{intcf_defn}
Let $\mcalc$ be a small category, and let $\mcalf \colon \mcalc \lra \text{Cat}$ be a covariant functor from $\mcalc$ to the category Cat of small categories. Define $\int_{\mcalc} \mcalf$ to be the category whose objects are pairs $(c, x)$ where $c \in \mcalc$ and $x \in \mcalf(c)$. A morphism $(c, x) \lra (c', x')$ consists of a pair $(f, g)$, where $f \colon c \lra c'$ is a morphism of $\mcalc$, and $g \colon \mcalf(f)(x) \lra x'$ is a morphism of $\mcalf(c')$.  The construction that sends $\mcalf \colon \mcalc \lra \text{Cat}$ to $\int_{\mcalc} \mcalf$ is called the \emph{Grothendieck construction}. 
\end{defn}

Here are two examples of the Grothendieck construction. The first one will be used in Subsection~\ref{sos_good_subsection}, while the second will be used in Section~\ref{poly_section}. 

\begin{expl} \label{intcf_expl1}
Let $\mcald_0 \hra \mcald_1  \hra \mcald_2 \cdots $ be an increasing inclusion of small categories. Define $\mcalc$ to be the category $\{0 \ra 1 \ra 2 \ra \cdots \}$, and $\mcalf \colon \mcalc \lra \text{Cat}$ as $\mcalf(i) = \mcald_i$. Then one can see that 
\[
 \int_{\mcalc} \mcalf  = \mcalc \times \left(\bigcup_{i=0}^{\infty} \mcald_i\right). 
 \]
\end{expl}

\begin{expl} \label{intcf_expl2}
Let $k \geq 0$, and let $\mcalc$ be the category defined as 
\[
\mcalc = \left\{S \subseteq \{0, \cdots, k\}\ \ \text{such that} \   \  S \neq \emptyset \right\}. 
\]
Given two objects $S, T \in \mcalc$, there exists a morphism from $S$ to $T$ if  $T \subseteq S$.   If $U$ is an open subset of $M$, and $A_0, \cdots, A_k$ are pairwise disjoint closed subsets of $U$, we let $\Omega$ denote the basis (for the topology of $U$) where an element is a subset $B$ diffeomorphic to an open  ball such that $B$ intersects at most one $A_i$. In other words, if one introduces the notation $U(S):= U \backslash \cup_{i \in S} A_i$, then 
\[
\Omega := \left\{B \subseteq U| \; \text{$B$ is diffeomorphic to an open ball and $B \subseteq U\left(\{0, \cdots, \hat{i}, \cdots, k \}\right)$}  \right\},
\]
where the \lq\lq hat\rq\rq{} means taking out.  Certainly $\Omega$ is a good basis (see Definition~\ref{gb_defn}).  Now, for $V \in \mcalo(U)$ we  let $\Omega_k(V) \subseteq \ok(V)$ denote the full  subposet whose objects are  disjoint unions of at most $k$ elements from $\Omega$. Define 
\[
\mcalf \colon \mcalc \lra \text{Cat}  \quad \text{as} \quad  \mcalf(S) := \Omega_k(U(S)).
\]
 Clearly, one has  $\mcalf(S) \subseteq \mcalf(T)$ whenever $T \subseteq S$. So $\mcalf(S \ra T)$ is just the inclusion functor.  One can then consider the category $\int_{\mcalc} \mcalf$ (see Definition~\ref{intcf_defn}), which can be described as follows. An object of that category is a pair $(S, V)$ where $\emptyset \neq S \subseteq \{0, \cdots, k\}$, and $V \subseteq U \backslash \cup_{i \in S} A_i$ is the disjoint union of at most $k$ elements from $\Omega$. There exists a morphism $(S, V) \lra (T, W)$ if and only if $T \subseteq S$ and $V \subseteq W$. 
\end{expl}

\begin{thm}\cite{cha_sch01}   \label{cha_sch_thm}
Let $\mcalm$ be a simplicial model category. Let $\mcalc$ be a small category and let $\mcalf \colon \mcalc \lra \text{Cat}$ be a covariant functor.  Consider a collection  $\{G_c \colon \colon \mcalf(c) \lra \mcalm\}_{c \in \mcalc}$ of functors such that for any $f \colon c \lra c'$ in $\mcalc$ the following triangle commutes.
\[
\xymatrix{\mcalf(c') \ar[rr]^-{G_{c'}}  \ar[d]_-{\mcalf(f)}  &  & \mcalm \\
           \mcalf(c) \ar[rru]_-{G_c}  &  &      }
\]
Then the canonical map 
\[
 \beta \colon \underset{(c, x) \in \int_{\mcalc} \mcalf}{\text{holim}} \; G_c(x) \lra \underset{c \in \mcalc}{\text{holim}} \; \underset{x \in \mcalf(c)}{\text{holim}} \; G_c(x)
\]
is a weak equivalence (see Definition~\ref{intcf_defn}). 
\end{thm} 

\begin{proof}
This is the dual of \cite[Theorem 26.8]{cha_sch01}.
%\footnote{Also here, is the definition of holim or hocolim in \cite{cha_sch01} equivalent to ours? I guess yes}. 
\end{proof}

\subsection{Special open sets and good cofunctors}   \label{sos_good_subsection}

The goal of this subsection is to prove Theorem~\ref{sos_thm} and Theorem~\ref{good_thm} announced at the beginning of Section~\ref{sos_good_section}.

To prove Theorem~\ref{sos_thm} we will need Lemma~\ref{fqf_lem} below. First we need to introduce some notation. For $U \in \om, q \geq 0$ we let $\bkq(U)$ denote the poset whose objects are strings $W_0 \ra \cdots \ra W_q$ of $q$ composable morphisms in $\bk(U)$ such that $W_i \ra W_{i+1}$ is an isotopy equivalence for all $i$. A morphism from $W_0 \ra \cdots \ra W_q$ to $W'_0 \ra \cdots \ra W'_q$ consists of a collection $f = \{f_i \colon W_i \lra W'_i\}_{i=0}^{q}$ of morphisms of  $\bk(U)$ such that all the obvious squares commute. Now let $F \colon \bkm \lra \mcalm$ be an objectwise fibrant cofunctor. Define a new cofunctor $\fhb^q \colon \bkq (U) \lra \mcalm$ as 
\begin{eqnarray} \label{fhq}
\fhb^q (W_0 \ra \cdots \ra W_q) := F(W_0).
\end{eqnarray}
Also define $\fhb^{!q} \colon \om \lra \mcalm$ as 
\begin{eqnarray} \label{fhsq}
\fhb^{!q} (U) := \underset{\bkq (U)}{\text{holim}} \; \fhb^q.  
\end{eqnarray}

\begin{rmk} \label{fsp_bkp_rmk}
\begin{enumerate} 
\item[(i)] As in Remark~\ref{fsp_rmk}, the collection $\fhb^{!\bullet} (U) = \{\fhb^{!q} (U)\}_{q \geq 0}$ is a cosimplicial object in $\mcalm$ for all $U \in \om$. 
\item[(ii)] Recall the poset $\bkp (U)$ and the functor $\ftb^p$ from Definition~\ref{bkp_defn} and Definition~\ref{fsbp_defn} respectively. Also recall $\pid (-)$ from Definition~\ref{cr_defn}. Then one can easily  see that 
\[
\Pi^q \ftb^p = \Pi^p \fhb^q
\]
since the set of $q$-simplices of the nerve $N(\bkp (U))$ is equal to the set of $p$-simplices of the nerve $N(\bkq (U))$. 
\end{enumerate}
\end{rmk}

\begin{lem} \label{fqf_lem}
Let $\mcalm$ be a simplicial model category. For $U \in \om$, consider the functor $\theta \colon \bk(U) \lra \bkq (U)$ defined as $\theta(V) := V \ra \cdots \ra V$. Then the canonical map 
\begin{eqnarray} \label{fqf_eq}
[\theta; \fhb^q] \colon \fhb^{!q} (U) \lra F^{!}_{\mcalb} (U)
\end{eqnarray}
is a weak equivalence.  Furthermore this map is natural in $U$. 
\end{lem}

\begin{proof}
Let $W_0 \ra \cdots \ra W_q \in \bkq (U)$. The under category $(W_0 \ra \cdots \ra W_q) \downarrow \theta$ is contractible since it has an initial object, namely $(W_q, f)$ where $f$ is the obvious map from $W_0 \ra \cdots \ra W_q$ to $W_q \ra \cdots \ra W_q$. So $\theta$ is homotopy right cofinal, and therefore the map $[\theta; \fhb^q]$ is a weak equivalence by Theorem~\ref{htpy_cofinal_thm}. The naturality of that map in $U$ is readily checked. 
\end{proof}

We are now ready to prove Theorem~\ref{sos_thm}.  

\begin{proof}[Proof of Theorem~\ref{sos_thm}]
In the following proof we will work with $\mcalb'=\mcalo$. Notice that one can perform exactly the same proof with any good basis $\mcalb'$ containing $\mcalb$. 

Let $U \in \om$. Recall $\fsbp$ from Definition~\ref{fsbp_defn}. We will show that the objects $\underset{[p] \in \Delta}{\text{holim}} \; \ftb^{!p} (U)$ and $F^{!}_{\mcalb} (U)$ are connected by a zigzag of natural weak equivalences. Consider the following diagram. 
\[
\xymatrix{\hpd \; \hbkp \; \ftb^p \ar[r]^-{\cong}  & \hpd \;  \tot \; \pid \ftb^p \ar[r]^-{\sim} & \hpd \; \hqd \; \Pi^q \ftb^p \ar[r]^-{\sim} & \hqd \; \hpd \; \Pi^q \ftb^p  \\
               \hpd \; \hokp \;  \ftpo \ar[u] \ar[r]_-{\cong} & \hpd \; \tot \; \pid \ftpo \ar[u] \ar[r]_-{\sim} & \hpd \; \hqd \; \Pi^q \ftpo \ar[u]  \ar[r]_-{\sim}  & \hqd \; \hpd \; \Pi^q \ftpo \ar[u]_-{\lambda} }
\] 
\begin{enumerate}
\item[$\bullet$] In the first row
       \begin{enumerate}
       \item[-] the first map is the homotopy limit of the map (\ref{phic_map}),
       \item[-] the second is the homotopy limit of the Bousfield-Kan map from Corollary~\ref{hir_coro}, and 
       \item[-] the third is provided by the Fubini Theorem~\ref{fubini_thm}. 
       \end{enumerate}
 \item[$\bullet$] The maps in the second row are obtained in the similar way  since  $\mcalb$ is a subposet of $\mcalo$. 
 \item[$\bullet$] The lefthand vertical map is nothing but the homotopy limit of $[\theta; \ftpo]$, where $\theta \colon \bkp (U) \hra \okp (U)$ is just the inclusion functor. 
 \item[$\bullet$] The three others are the canonial ones induced by the map $\beta^{\bullet}_{\okp (U)}$ from Remark~\ref{betax_rmk}.      
\end{enumerate}
Certainly the lefthand square commutes by Proposition~\ref{comm_prop}. The middle one commutes by the fact that the Bousfield-Kan map $\tot \; Z^{\bullet} \lra \underset{\Delta}{\text{holim}} \; Z^{\bullet}$ is natural in $Z^{\bullet}$. The third square commutes since the map in the Fubini  Theorem~\ref{fubini_thm} is also natural.  So the above diagram is commutative. Therefore, since the first vertical map is a weak equivalence by Proposition~\ref{sosp_prop} and Theorem~\ref{fib_cofib_thm}, it follows that the last one, $\lambda$, is also a weak equivalence. 

Similarly the following diagram (in which the equality comes from Remark~\ref{fsp_bkp_rmk}) is commutative as well. 
\[
\xymatrix{ \hqd \; \hpd \; \Pi^q \ftpb  \ar@{=}[r]  &  \hqd \; \hpd \; \Pi^p \fhb^q  & \hqd \; \tot \; \pid \fhb^q \ar[l]_-{\sim} & \hqd \;  \hbkq  \fhb^q \ar[l]_-{\cong} \\
                 \hqd \; \hpd \; \Pi^q \fto  \ar@{=}[r]  \ar[u]_-{\lambda}^-{\sim} & \hqd \; \hpd \; \Pi^p \fho^q  \ar[u] &  \hqd \; \tot \; \pid \fho^q \ar[l]^-{\sim}   \ar[u] & \hqd \; \hokq \; \fho^q \ar[l]^-{\cong}  \ar[u]_-{\varphi}.   }
\]
So the map $\varphi$ is a weak equivalence. 

Now consider the following  square. 
\[
\xymatrix{\hqd \;  \fhb^{!q} (U) \ar[rr]^-{\sim}  &   &  \fsb (U)  \\
                \hqd \; \fho^{!q} (U) \ar[rr]_-{\sim} \ar[u]^-{\varphi}_-{\sim}  &    &  \fso (U), \ar[u] }
\]
where the top horizontal arrow is the homotopy limit of the map (\ref{fqf_eq}), which is itself a weak equivalence by Lemma~\ref{fqf_lem}. In similar fashion the bottom horizontal map is a weak equivalence. The lefthand vertical map is the above map $\varphi$, which is a weak equivalence. So, since the square commutes, it follows that the righthand vertical map is also a weak equivalence. We thus obtain the desired result. 
\end{proof}

To prove Theorem~\ref{good_thm} we will need the following lemma. 

\begin{lem} \label{gd_lem}
Let $U \in \om$, and let $\bkb (U) \subseteq \bk(U)$ be the full subposet defined as 
\begin{eqnarray} \label{bkb_eq}
\bkb (U) = \{V \in \bk(U)| \; \overline{V} \subseteq U\}.
\end{eqnarray}
Here $\overline{V}$ stands for the closure of $V$. Let $\mcalm$ be a simplicial model category.  Consider an isotopy cofunctor $F \colon \bk (U) \lra \mcalm$. Then the canonical map 
\begin{eqnarray} \label{good_eq}
[\theta; F] \colon \underset{V \in \bk(U)}{\text{holim}} \; F(V) \lra \underset{V \in \bkb(U)}{\text{holim}} \; F(V),
\end{eqnarray}
induced by the inclusion functor $\theta \colon \bkb (U) \lra \bk (U)$, is a weak equivalence. 
\end{lem}

\begin{proof}
Let $\overline{\mcalb}$ be the following basis for the topology of $U$.
\[
\overline{\mcalb} = \{B \subseteq U| \; \text{$B \in \mcalb$ and $\overline{B} \subseteq U$}\}.
\]
Certainly $\overline{\mcalb}$ is a good basis (see Definition~\ref{gb_defn}). One can easily see that each object of $\bkb (U)$ is the disjoint union of at most $k$ elements from $\overline{\mcalb}$. So by Theorem~\ref{sos_thm}, the map $[\theta; F]$ is  a weak equivalence. 
\end{proof}

\begin{proof}[Proof of Theorem~\ref{good_thm}]
We begin with part (a) of goodness. Let $U, U' \in \om$ such that $U \subseteq U'$. Assume that the inclusion map $U \hra U'$ is an isotopy equivalence. Then the canonical map
\[
\underset{\bkp (U')}{\text{holim}} \; \ftb^p \lra \hbkp \; \ftb^p
\]
is a weak equivalence by Proposition~\ref{isop_prop}. Now, by replacing 
\begin{enumerate}
\item[$\bullet$] $\okp(U)$ by $\bkp (U')$, 
\[
\fto^p \colon \okp (U) \lra \mcalm \qquad \text{by} \qquad \ftpb \colon \bkp (U') \lra \mcalm,
\] 
\[
\fho^q \colon \okq (U) \lra \mcalm \qquad \text{by} \qquad \fhb^q \colon \bkq (U') \lra \mcalm,
\]
\item[$\bullet$] $\okq (U)$ by $\bkq (U')$, $\fho^{!q} (U)$ by $\fhb^{!q} (U')$, and $\fso (U)$ by $\fsb(U')$, 
\end{enumerate}
in the proof of Theorem~\ref{sos_thm}, we deduce that the canonical  map $\fsb(U') \lra \fsb (U)$ is a weak equivalence. This proves part (b) from Definition~\ref{isotopy_cof_defn}. Part (a) from the same definition follows immediately from the fact that $F \colon \bk(M) \lra \mcalm$ is an isotopy cofunctor by assumption, and from Theorem~\ref{fib_cofib_thm}.  

Now we show part (b) of goodness. Let $U_0 \ra U_1 \ra \cdots$ be a string of inclusions of $\om$. Consider the following commutative square. 
\[
\xymatrix{ \underset{V \in \bk (\cup_i U_i)}{\text{holim}} \; F(V) \ar[rr] \ar[d]_-{\sim}   &   &  \underset{i}{\text{holim}} \; \underset{V \in \bk (U_i)}{\text{holim}} \; F(V) \ar[d]^-{\sim} \\
 \underset{V \in \bkb (\cup_i U_i)}{\text{holim}} \; F(V) \ar[rr]_-{\sim}  &  &  \underset{i}{\text{holim}} \; \underset{V \in \bkb (U_i)}{\text{holim}} \; F(V). }
\]
\begin{enumerate}
\item[$\bullet$] Both vertical maps come from (\ref{good_eq}), and therefore are weak equivalences by Lemma~\ref{gd_lem}. 
\item[$\bullet$] The bottom horizontal map is a weak equivalence by the following reason. Consider the data from Example~\ref{intcf_expl1}, and set $\mcald_i := \bkb (U_i)$. Then it is straightforward to see  that 
\begin{eqnarray} \label{int_eq}
\int_{\mcalc} \mcalf = \mcalc \times \left(\bigcup_i \mcald_i\right) = \bkb (\cup_i U_i). 
\end{eqnarray}
The first equality is obvious, while the second one comes from the definition of $\bkb (-)$ (see (\ref{bkb_eq})). Furthermore the canonical map 
\[
\underset{(i, V) \in \int_{\mcalc} \mcalf}{\text{holim}} \; F(V) \lra \underset{i \in \mcalc}{\text{holim}} \; \underset{V \in \mcald_i}{\text{holim}} \; F(V)
\]
is a weak equivalence by Theorem~\ref{cha_sch_thm}. But, by (\ref{int_eq}),  this latter map is nothing but the map we are interested in. 
\end{enumerate}
Hence the top horizontal map is a weak equivalence, and this completes the proof. 
\end{proof}

\section{Polynomial cofunctors}  \label{poly_section}

The goal of this section is to prove Theorem~\ref{main_thm} announced in the introduction.  We will need three preparatory lemmas: Lemma~\ref{cofinal_lem}, Lemma~\ref{poly_lem}, and Lemma~\ref{charac_lem}. The two latter ones are important themselves. 

Let us begin with the following definition. 

\begin{defn} \label{poly_defn}
A cofunctor $F \colon \om \lra \mcalm$ is called \emph{polynomial of degree $\leq k$} if for every $U \in \om$ and pairwise disjoint closed subsets $A_0, \cdots, A_k$ of $U$, the canonical map 
\[
F(U) \lra \underset{S \neq \emptyset}{\text{holim}} \; F(U \backslash \cup_{i \in S} A_i)
\] 
is a weak equivalence. Here $S \neq \emptyset$ runs over the power set of $\{0, \cdots, k\}$. 
\end{defn}

\begin{lem} \label{cofinal_lem}
Consider the data from Example~\ref{intcf_expl2}. Then the functor  $\theta \colon \int_{\mcalc} \mcalf \lra \Omega_k(U)$, defined as $\theta (S, V) = V$, is homotopy right cofinal (see Definition~\ref{cofinal_defn}).   
\end{lem}

\begin{proof}
First of all, let us consider the notation ($\mcalc, U(S), \Omega, \Omega_k(U), \mcalf$) introduced in Example~\ref{intcf_expl2}. One has the following properties. 
\begin{enumerate}
\item[(a)] $\mcalf(S \cup T) = \mcalf(S) \cap \mcalf(T)$ for any $S, T \in \mcalc$;
\item[(b)] for any $X \in \Omega_k(U)$ there exists $j \in \{0, \cdots, k\}$ such that $X \cap A_j = \emptyset$.
\end{enumerate}
The first property follows directly from the definitions. The second comes from the following three facts: (i) By definition, each element of $\Omega$ intersects at most one of the $A_i$'s. (ii) $X$ is the disjoint union of at most $k$ elements from $\Omega$. (iii) The cardinality of the set $\{A_0, \cdots, A_k\}$ is $k+1$, which is greater than the number of components of $X$.  This property is nothing but the pigeonhole principle. 

Now let $V \in \Omega_k(U)$. We have to prove that the under category (see Definition~\ref{uc_defn}) $V \downarrow \theta$ is contractible. It suffices to show that it admits an initial object.  Consider the pair $(S, V)$ where 
\[
S = \left\{i \in \{0, \cdots, k\} |  \  V \cap A_i = \emptyset \right\}.
\]
Certainly $S \neq \emptyset$ by the property (b). So $S$ is an object of $\mcalc$. Moreover one can see that $V \in \cap_{i \in S} \mcalf(\{i\})$. This amounts to saying that $V \in \mcalf(S)$ since $\cap_{i \in S} \mcalf(\{i\}) = \mcalf(\cup_{i \in S} \{i\})$ by (a). So $(S, V) \in \int_{\mcalc} \mcalf$. Hence  the pair $((S, V), id_V)$ is an object of $V \downarrow \theta$. We claim that this latter object is an initial object of $V \downarrow \theta$.  To prove the claim, let $((T, W), V \hra W)$ be another object of  $V \downarrow \theta$. Since $V \subseteq W$, it follows that $\{i | \  V \cap A_i \neq \emptyset \}$ is a subset of $\{i | \ W \cap A_i \neq \emptyset\}$. This implies that 
\[
\{i | W \cap A_i  =  \emptyset\} \subseteq \{i | V \cap A_i  =  \emptyset\} = S. 
\]
Furthermore, $T$ is a subset of $\{i | W \cap A_i  =  \emptyset\} $ since $W \in \mcalf(T) = \Omega_k(U \backslash (\cup_{i \in T} A_i))$. So $T \subseteq S$, and therefore there is a unique morphism from $((S, V), id_V)$ to $((T, W), V \hra W)$ in the under category $V \downarrow \theta$. This completes the proof.   
\end{proof}

\begin{lem}  \label{poly_lem}
Let $\mcalm$ be a simplicial model category.  
\begin{enumerate}
\item[(i)] Let $\mcalo$ and $\okm$ as in Example~\ref{fso_expl}.   Let $F \colon \okm \lra \mcalm$ be an isotopy cofunctor (see Definition~\ref{isotopy_cof_defn}). Then the cofunctor $F^{!}_{\mcalo} \colon  \om \lra \mcalm$ (see Example~\ref{fso_expl}) is polynomial of degree $\leq k$. 
\item[(ii)] Let $\mcalb$ and $\mcalb_k(M)$ as in Definition~\ref{fsb_defn}.  Let $F \colon \bkm \lra \mcalm$ be an isotopy cofunctor. Then the cofunctor $F^{!}_{\mcalb} \colon \om \lra \mcalm$ (see Definition~\ref{fsb_defn}) is polynomial of degree $\leq k$. 
\end{enumerate}
\end{lem}

\begin{proof}
We begin with the first part. First let us consider again the notation ($\mcalc, U(S), \Omega, \Omega_k(U), \mcalf$) introduced in Example~\ref{intcf_expl2}.  We will first show that the canonical map 
\[
\Phi_{\Omega} \colon \underset{V \in \Omega_k(U)}{\text{holim}} \; F(V) \lra  \underset{S \in \mcalc}{\text{holim}}\; \underset{V \in \mcalf(S)}{\text{holim}} \; F(V)
\]
is a weak equivalence. Next, by using the fact that $\Omega$ is another basis (for the topology of $U$) contained in $\mcalo|U$, we will deduce that the canonical map 
\[
\Phi_{\mcalo} \colon \underset{V \in \ok(U)}{\text{holim}} \; F(V) \lra  \underset{S \in \mcalc}{\text{holim}}\; \underset{V \in \ok(U(S))}{\text{holim}} \; F(V)
\]
is also a weak equivalence.

One can see that the map $\Phi_{\Omega}$ factors through $\underset{(S, V) \in \int_{\mcalc} \mcalf}{\text{holim}} \; F(V)$. That is, there is a commutative triangle
\[
\xymatrix{ \underset{V \in \Omega_k(U)}{\text{holim}} \; F(V) \ar[rr]^-{\Phi_{\Omega}}  \ar[rd]_-{\alpha}  &     &   \underset{S \in \mcalc}{\text{holim}}\; \underset{V \in \mcalf(S)}{\text{holim}} \; F(V) \\
                             &  \underset{(S, V) \in \int_{\mcalc} \mcalf}{\text{holim}} \; F(V) , \ar[ru]_-{\beta}   }
\]
where 
\begin{enumerate}
\item[$\bullet$] $\alpha$ is nothing but $[\theta; F]$ (see (\ref{thetax})), where $\theta \colon \intcf \lra \Omega_k(U)$ is the map from Lemma~\ref{cofinal_lem}. 
\item[$\bullet$] $\beta$ is the canonical map from  Theorem~\ref{cha_sch_thm}. 
\end{enumerate}
Since  $\alpha$ is a weak equivalence by Lemma~\ref{cofinal_lem} and Theorem~\ref{htpy_cofinal_thm}, and since $\beta$ is a weak equivalence by Theorem~\ref{cha_sch_thm}, it follows that $\Phi_{\Omega}$ is  a weak equivalence as well.

Now consider the following commutative square induced by the inclusion $\Omega_k(U) \hra \ok(U)$.  
\[
\xymatrix{ \underset{V \in \Omega_k(U)}{\text{holim}} \; F(V) \ar[rr]^-{\Phi_{\Omega}}_-{\sim}  &   &  \underset{S \in \mcalc}{\text{holim}}\; \underset{V \in \mcalf(S)}{\text{holim}} \; F(V)  \\
 \underset{V \in \mcalo_k(U)}{\text{holim}} \; F(V) \ar[u]^-{\sim} \ar[rr]_-{\Phi_{\mcalo}}  &      &     \underset{S \in \mcalc}{\text{holim}}\; \underset{V \in \ok(U(S))}{\text{holim}} \; F(V) \ar[u]_-{\sim} }
\] 
Since the lefthand vertical map is a weak equivalence by Theorem~\ref{sos_thm}, and since the righthand vertical map is also a weak equivalence by Theorems~\ref{sos_thm}, \ref{fib_cofib_thm} (remember that $U(S) := U\backslash \cup_{i \in S}A_i$ and $\mcalf(S) := \Omega_k(U(S))$), it follows that $\Phi_{\mcalo}$ is a weak equivalence as well.  And this proves part (i). 

Now we prove part (ii). Let $F \colon \bkm \lra \mcalm$ be an isotopy cofunctor. We have to show that $\fsb$ is polynomial of degree $\leq k$. First, consider the cofunctor $G \colon \okm \lra \mcalm$ defined as $G(U) := \fsb |  \okm$, the restriction of $\fsb$ to $\okm$.  Certainly $G$ is an isotopy cofunctor since $\fsb$ is good by Theorem~\ref{good_thm}.  This implies by the first part that the cofunctor $\gso \colon \om \lra \mcalm$ is polynomial of degree $\leq k$. Moreover the canonical map $\gso \lra (G | \bk(M))^{!}_{\mcalb}$ is a weak equivalence by Theorem~\ref{sos_thm}. So $(G | \bk(M))^{!}_{\mcalb}$ is also polynomial of degree $\leq k$. Now, since the canonical map $F \lra G | \bk(M)$ is a weak equivalence by Proposition~\ref{fsrik_prop}, it follows that the induced map $\fsb \lra (G | \bk(M))^{!}_{\mcalb}$ is a weak equivalence as well. And therefore $\fsb$ is polynomial of degree $\leq k$. This proves the lemma. 
\end{proof}

\begin{lem} \label{charac_lem}
Let $\mcalm$ be a simplicial model category.  
\begin{enumerate}
\item[(i)] Let $F, G \colon \om \lra \mcalm$ be good and polynomial cofunctors of degree $\leq k$. Let $\eta \colon F \lra G$ be a natural transformation such that for any $U \in \okm$ the component $\eta[U] \colon F(U) \lra G(U)$ is a weak equivalence. Then for any $U \in \om$ the map $\eta[U]$ is a weak equivalence. 
\item[(ii)] Let $F, G \colon \om \lra \mcalm$ be good and polynomial cofunctors of degree $\leq k$.  Let $\mcalb$ and $\bkm$ as in Definition~\ref{fsb_defn}.  Consider a natural transformation  $\eta \colon F \lra G$ such that for any $U \in \bkm$ the component $\eta[U] \colon F(U) \lra G(U)$ is a weak equivalence. Then for any $U \in \om$ the map $\eta[U]$ is a weak equivalence. 
\end{enumerate}
\end{lem}

\begin{proof}
The first part can be proved by following exactly the same steps as those of the proof of Theorem 5.1 from \cite{wei99}. Now we prove the  second part. Let $U \in \okm$. Since $\mcalb$ is a basis for the topology of $M$, there exists $V \in \bkm$ contained in $U$ and such that the inclusion $V \hra U$ is an isotopy equivalence.  Applying $\eta$ to $V \hra U$, we get the following commutative square. 
\[
\xymatrix{F(U) \ar[r]^-{\sim} \ar[d]_-{\eta[U]}  &  F(V) \ar[d]^-{\sim}_-{\eta[V]}  \\
            G(U)  \ar[r]_-{\sim}   &  G(V).}
\]
The top and the bottom maps are weak equivalences since $F$ and $G$ are good by hypothesis. The righthand vertical map is a weak equivalence  by assumption. So the lefthand vertical map is also a weak equivalence. Hence $\eta[U]$ is a weak equivalence for every $U \in \okm$. Now the desired result follows from the first part. 
\end{proof}

We are now ready to prove the main result of the paper. 

\begin{proof}[Proof of Theorem~\ref{main_thm}]
Assume that $F \colon \om \lra \mcalm$ is good and polynomial of degree $\leq k$. Define $G$ to be the restriction of $F$ to $\bkm$. That is, $G:= F|\bk(M)$. Since $F$ is good, and then in particular an isotopy cofunctor, it follows that $G$ is an isotopy cofunctor as well. Now we want to show that the canonical map $\eta \colon F \lra G^{!}$ is a weak equivalence. Let $U \in \om$. One can rewrite $\eta[U] \colon F(U) \lra G^{!}(U)$ as the composition
\[
\eta[U] \colon \xymatrix{ F(U) \ar[r]^-{f}    &  \underset{V \in \mcalo(U)}{\text{holim}} \; F(V) \ar[r]^-{g}  &  G^{!} (U), }
\] 
where $f$ comes from the fact that $U$ is the terminal object of $\mcalo(U)$. The same fact allows us to conclude that $f$ is a weak equivalence. The map $g$ is nothing but $[\theta; F]$, where $\theta \colon \bk(U) \lra \mcalo(U)$ is just the inclusion functor. Now assume  $U \in \bk(M)$. Then for every $V \in \mcalo(U)$ the under category $V \downarrow \theta$ is contractible since it has a terminal object, namely $(U, V \hra U)$. Therefore, by Theorem~\ref{htpy_cofinal_thm}, the map $g$ is a weak equivalence. This implies that $\eta[U]$ is a weak equivalence when $U \in \bkm$.  So, by Lemma~\ref{charac_lem}, $\eta[U]$ is also a weak equivalence for any $U \in \om$. 

Conversely, assume that $G:=F|\bk(M)$ is an isotopy cofunctor and that the canonical map $\eta \colon F \stackrel{\sim}{\lra} G^{!}$ be a  weak equivalence. By Theorem~\ref{good_thm} and Lemma~\ref{poly_lem} the cofunctor $G^{!}$ is good and polynomial of degree $\leq k$, which  proves the converse. We thus obtained the desired result.   
\end{proof}

\section{Homogeneous cofunctors}  \label{hc_section}

The goal of this section is to prove Theorem~\ref{main2_thm} (announced in the introduction), which roughly says that the category of homogeneous cofunctors $\om \lra \mcalm$ of degree $k$ is weakly equivalent to the category of 
linear cofunctors $\mcalo(F_k(M)) \lra \mcalm$. We begin with three definitions. Next we prove Lemma~\ref{homo_lem}, which is the key lemma here, and which roughly states that homogeneous cofunctors of degree $k$ are determined by their values on open subsets diffeomorphic to the disjoint union of exactly $k$ balls. Note that this lemma is also a useful result in its own right, and its proof is based on the results we obtained in Section~\ref{sos_good_section} and Section~\ref{poly_section}.

\begin{defn}
Let $\mcalm$ be a simplicial model category, and let $F \colon \om \lra \mcalm$ be a  cofunctor. The \emph{$k$th polynomial approximation} to $F$, denoted $T_kF$, is the cofunctor $T_k F \colon \om \lra \mcalm$ defined as 
\[
T_k F (U) :=  \underset{V \in \ok (U)}{\text{holim}} F(V).
\] 
\end{defn}

\begin{defn} \label{hc_defn}
Let $\mcalm$ be a simplicial model category that has a terminal object denoted $0$. 
\begin{enumerate}
\item[(i)] A cofunctor $F \colon \om \lra \mcalm$ is called \emph{homogeneous of degree} $k$ if it satisfies the following three conditions:
\begin{enumerate}
\item[(a)] $F$ is a good cofunctor (see Definition~\ref{good_defn});
\item[(b)] $F$ is polynomial of degree $\leq k$ (see Definition~\ref{poly_defn}); 
\item[(c)] The unique map $T_{k-1} F (U) \lra 0$ is a weak equivalence for every $U \in \om$. 
\end{enumerate} 
\item[(ii)] A \emph{linear cofunctor} is a homogeneous cofunctor of degree $1$. 
\end{enumerate}
\end{defn}

The category of homogeneous cofunctors of degree $k$ and natural transformations will be denoted by $\fko$.

\begin{defn} \label{we_defn}
Let $\mcalc$ and $\mcald$ be categories both equipped with a class of maps called weak equivalences. 
\begin{enumerate}
\item[(i)] We say that two functors  $F, G \colon \mcalc \lra \mcald$ are  \emph{weakly equivalent}, and we denote $F \simeq G$,  if they are connected by a zigzag of objectwise weak equivalences. 
\item[(ii)] A functor $F \colon \mcalc \lra \mcald$ is said to be a \emph{weak equivalence} if it satisfies the following two conditions.
   \begin{enumerate}
     \item[(a)] $F$ preserves weak equivalences.
		\item[(b)] There is a functor $G \colon \mcald \lra \mcalc$ such that $FG$ and $GF$ are both weakly equivalent to the identity. The functor $G$ is also required to preserve weak equivalences. 
		\end{enumerate} 
\item[(iii)] We say that $\mcalc$ is weakly equivalent  to $\mcald$, and we denote $\mcalc \simeq \mcald$, if there exists a zigzag of weak equivalences between $\mcalc$ and $\mcald$. 
\end{enumerate}
\end{defn}

\begin{rmk}
By Definition~\ref{we_defn}, it follows that if two categories $\mcalc$ and $\mcald$ are weakly equivalent, then their localizations with respect to weak equivalences are equivalent in the classical sense. Note that no model structure is required on $\mcalc$ and $\mcald$. So our notion of weak equivalences between categories is not comparable, in general, with the well known notion of Quillen equivalence.  
\end{rmk}

As mentioned earlier the following lemma is the key ingredient in proving Theorem~\ref{main2_thm}. 

\begin{lem} \label{homo_lem}
Let $\mcalb$ be a good basis (see Definition~\ref{gb_defn}) for the topology of $M$. Let $\bku (M) \subseteq \om$ denote the  subposet whose objects are disjoint unions  of exactly $k$ elements from $\mcalb$, and whose morphisms are isotopy equivalences.  Let $\mcalm$ be a simplicial model category.  Assume that $\mcalm$ has a zero object $0$ (that is, an object which is both terminal an initial). 
\begin{enumerate}
	\item[(i)]  Then the category $\fko$ of homogeneous cofunctors of degree $k$ (see Definition~\ref{hc_defn}) is weakly equivalent (in the sense of Definition~\ref{we_defn}) to the category $\mcalf (\bku (M); \mcalm)$ of isotopy cofunctors $\bku (M) \lra \mcalm$ (see Definition~\ref{isotopy_cof_defn}). That is,
	\[
	\fko \simeq \mcalf(\bku(M); \mcalm). 
	\]
	\item[(ii)] For $A \in \mcalm$ we have the weak equivalence 
	\[
	\mcalf_{kA}(\om; \mcalm) \simeq \mcalf_A(\bku(M); \mcalm),
	\]
	where $\mcalf_{kA}(\om; \mcalm)$ is the category from Theorem~\ref{main2_thm} and $\mcalf_A(\bku(M); \mcalm)$ denotes the category of isotopy cofunctors $F \colon \bku (M) \lra \mcalm$ such that $F(U) \simeq A$ for every $U \in \bku (M)$.  
\end{enumerate}
\end{lem}

\begin{proof}
We will prove the first part; the proof of the second part is similar. The idea of the proof is to define a new category and show that it is weakly equivalent to both $\mcalf(\bku(M); \mcalm)$ and $\fko$. To define that category, let us first recall the notation $\bkm$ from Definition~\ref{fsb_defn}. Define  $\mcalf_k(\bkm; \mcalm)$ to be the category  whose objects are isotopy cofunctors $F \colon \bkm \lra \mcalm$ such that the restriction to $\mcalb_{k-1} (M)$ is weakly equivalent to the constant functor at $0$. That is,
\begin{eqnarray}  \label{wd_cond}
\text{for all $U \in \mcalb_{k-1}(M)$,} \quad F(U) \simeq 0. 
\end{eqnarray}
Now consider the following diagram
\begin{eqnarray} \label{psii_phii}
\xymatrix{\mcalf(\bku(M); \mcalm) \ar@<1ex>[r]^-{\psi_1} &  \mcalf_k(\bkm; \mcalm) \ar@<1ex>[l]^-{\phi_1} \ar@<1ex>[r]^-{\psi_2} & \fko  \ar@<1ex>[l]^-{\phi_2} }
\end{eqnarray}
where the maps are defined as follows. 
\begin{enumerate}
\item[$\bullet$] $\phi_1$ is the restriction functor. That is, $\phi_1(F) = F|\bku(M)$.
\item[$\bullet$] $\phi_2$ is also the restriction functor. To see that it is well defined, let $F \colon \om \lra \mcalm$ be an object of $\fko$. We have to check that $F$ satisfies condition (\ref{wd_cond}). So let $U \in \mcalb_{k-1}(M)$. Recalling the notation \lq\lq$[-; -]$\rq\rq{} from Proposition~\ref{induced_holim_prop}, we have the following commutative diagram
\[
\xymatrix{\underset{V \in \mcalo_{k-1}(U)}{\text{holim}} \; F(V)  \ar[rr]^-{\sim}  \ar[d]_-{[\theta; F]}^-{\sim}  &  & 0  \\
              \underset{V \in \mcalb_{k-1}(U)}{\text{holim}} \; F(V)  \ar[rru]^-{\sim}  &  & F(U). \ar[ll]^-{\sim} \ar[u]    }
\]
Here $\theta \colon \mcalb_{k-1}(U) \hra  \mcalo_{k-1}(U)$ is the inclusion functor. The bottom horizontal map is a weak equivalence since $U$ is the terminal object of $\mcalb_{k-1}(U)$ (see Proposition~\ref{fsrik_prop}). The top one is a weak equivalence since $F$ is homogeneous of degree $k$. Since the lefthand vertical map is also a weak equivalence (by Theorem~\ref{sos_thm}), it follows that $F(U)$ is weakly equivalent to $0$. 
\item[$\bullet$]  $\psi_1$ is defined as 
\[
\psi_1(F)(U) := \left\{ \begin{array}{cc}
                       F(U)  & \text{if $U \in \bku (M)$ }  \\
											  0   & \text{otherwise,}
                        \end{array} \right.
\]
Certainly $\psi_1(F)$ satisfies (\ref{wd_cond}) and is an isotopy cofunctor. This latter assertion comes from the fact that if $U \subseteq U'$ is an isotopy equivalence, then $U$ and $U'$  definitely have the same number of connected components. On morphisms $\psi_1$ is defined in the obvious way. 
\item[$\bullet$]  $\psi_2$ is defined as $\psi_2(F) := \fsb$ (see Definition~\ref{fsb_defn}). On morphisms $\psi_2$ is defined by the fact that the homotopy right Kan extension is functorial. By Theorem~\ref{good_thm} and Lemma~\ref{poly_lem}, it is clear that $\psi_2(F)$ is good and polynomial of degree $\leq k$. To see that $\psi_2(F)$ satisfies condition (c) from Definition~\ref{hc_defn}, let $F \colon \bkm \lra \mcalm$ be an object of $\mcalf_k(\bkm; \mcalm)$. Consider the following commutative diagram
\[
\xymatrix{\underset{V \in \mcalo_{k-1}(U)}{\text{holim}} \; \fsb (V) \ar[r]^-{[\theta; \fsb]}_-{\sim} \ar[d] &  \underset{V \in \mcalb_{k-1}(U)}{\text{holim}} \; \fsb (V) \ar@{=}[r] & \underset{V \in \mcalb_{k-1}(U)}{\text{holim}} \; \underset{W \in \bk(V)}{\text{holim}} \; F(W) \ar[lld]^-{\sim}  \\    
  0 	&   &     \underset{V \in \mcalb_{k-1}(U)}{\text{holim}} \; F(V), \ar[ll]^-{\sim} \ar[u]_-{\sim}}
\]
where the righthand vertical map is induced by the canonical map $F(V) \lra \underset{W \in \bk(V)}{\text{holim}} \; F(W)$. Since $V$ belongs to  $\mcalb_{k-1}(U)$ it follows that $V$ is the terminal object of $\bk(V)$, and therefore  this latter map is a weak equivalence (by Proposition~\ref{fsrik_prop}).  The bottom horizontal map is a weak equivalence since $F$ belongs to $\mcalf_k(\bkm; \mcalm)$, and then satisfies (\ref{wd_cond}). Regarding the map $[\theta; \fsb]$, it is a weak equivalence by Theorem~\ref{sos_thm}. All this implies that the lefthand vertical map is a weak equivalence as well. So $\fsb$ satisfies condition  (c). 
\end{enumerate}
Certainly $\phi_1, \psi_1$ and $\phi_2$ preserve weak equivalences. The functor $\psi_2$ preserves weak equivalences as well by Theorem~\ref{fib_cofib_thm} and condition (a) from Definition~\ref{isotopy_cof_defn}. Moreover, it is clear that
 $\phi_1\psi_1 =id$ and  $\psi_1 \phi_1 \simeq id$. So the category $\mcalf(\bku(M); \mcalm)$ is weakly equivalent to the category  $\mcalf_k(\bkm; \mcalm)$. Furthermore, by using Propsosition~\ref{fsrik_prop}, one can easily come to $\phi_2\psi_2 \simeq id$ and $\psi_2\phi_2 \simeq id$. So the categories $\mcalf_k(\bkm; \mcalm)$ and $\fko$ are also weakly equivalent. This proves the lemma. 
\end{proof}

We are now ready to prove Theorem~\ref{main2_thm}. 

\begin{proof}[Proof of Theorem~\ref{main2_thm}]
We will prove the first part; the proof of the second part is similar. Recall the notation $F_k(M)$, which is that of the space of unordered configuration of $k$ points in $M$. The proof of part (i) follows from the following three weak equivalences: (\ref{we1}), (\ref{we2}), and (\ref{we3}). The first one
\begin{eqnarray} \label{we1}
\fko \simeq \mcalf (\bku(M); \mcalm),
\end{eqnarray}
is nothing but Lemma~\ref{homo_lem}~-(i). The second
\begin{eqnarray} \label{we2}
\mcalf (\bku(M); \mcalm) \cong \mcalf(\mcalb'^{(1)} (F_k(M)); \mcalm),
\end{eqnarray}
is actually an isomorphism where $\mcalb'$ is the basis for the topology of $F_k(M)$ whose elements are products of exactly $k$ elements from $\mcalb$. This isomorphism comes from the fact that $\bku(M) \cong \mcalb'^{(1)} (F_k(M))$. The last weak equivalence 
\begin{eqnarray} \label{we3}
\mcalf(\mcalb'^{(1)} (F_k(M)); \mcalm) \simeq \mcalf_1 (\mcalo(F_k(M)); \mcalm),
\end{eqnarray}
is again Lemma~\ref{homo_lem}~-(i). 
\end{proof}

\section{Isotopy cofunctors in general model categories}  \label{iso_cof_section}

This section is independent of previous ones, and its goal is to prove Theorem~\ref{iso_cof_thm} (announced in the introduction), which says that the cofunctor $F^{!}=\fso$ from Example~\ref{fso_expl} is an isotopy cofunctor provided that $F \colon \okm \lra \mcalm$ is an isotopy cofunctor (here $\mcalm$ is a general model category). This result is proved in Theorem~\ref{good_thm} when $\mcalm$ is a simplicial model category. To prove Theorem~\ref{good_thm} we used several results/properties (about homotopy limits in simplicial model categories) including Theorem~\ref{fubini_thm}, Proposition~\ref{comm_prop}.  This latter result involves the notion of totalization of a cosimplicial object, which does not  make sense  in a general model category. So the method we used before do not work here anymore.   In this section we present a completely different approach, but rather lengthy, that uses only two properties  of homotopy limits (see Theorem~\ref{fib_cofib_thmg} and Theorem~\ref{htpy_cofinal_thmg}).  That approach is inspired by our work in \cite{paul_don17}. For the plan of this section, we refer the reader to the table of contents and the outline given at the introduction.  For a faster run through the section, the reader could, after reading the Introduction, jump directly  to the   beginning of Section~\ref{iso_cof_subsection} to get a better idea of the proof of  Theorem~\ref{iso_cof_thm}.

\subsection{Homotopy limits in general model categories}  \label{holim_subsectiong}

This subsection recalls some useful properties of homotopy limits in general model categories. We also recall two results (Proposition~\ref{induced_holim_propg} and Proposition~\ref{fsrik_propg}) that will be used in next subsections.

Homotopy limits and colimits in general model categories are constructed in \cite{hir03, dhks04} by W. Dwyer, P. Hirschhorn, D. Kan, and J. Smith. They use the notion of \textit{frames}  that we now recall briefly. Let $\mcalm$ be a model category, and let $X$ be an object of $\mcalm$. A \textit{cosimplicial frame} on $X$ is a  \lq\lq cofibrant replacement\rq\rq{} (in the Reedy model category of cosimplicial objects in $\mcalm$) of the constant cosimplicial object at $X$ that satisfies certain properties. A \textit{simplicial frame} on $X$ is the dual notion. For a more precise definition we refer the reader to \cite[Definition 16.6.1]{hir03}. A \textit{framing} on $\mcalm$ is a functorial cosimplicial and simplicial frame on every object of $\mcalm$. A \textit{framed model category} is   a model category endowed with a framing (see also \cite[Definition 16.6.21]{hir03}). A typical example of a framed model category is any simplicial model category as we considered in previous sections. 

\begin{rmk}
In \cite[Theorem 16.6.9]{hir03} it is proved that there exists a framing on any model category. It is also proved that two any framings are \lq\lq weakly equivalent\rq\rq{} \cite[Theorem 16.6.10]{hir03}. Throughout this section, $\mcalm$ is a model category endowed with a fixed framing.  
\end{rmk}

Using the notion of framing, one can define the homotopy limit and colimit of a diagram in $\mcalm$. We won't give that definition here since it is not important for us (the reader who is interested in that definition can find it in \cite[Definition 19.1.2 and Definition 19.1.5]{hir03}). All we need are some properties  of  that homotopy limit and colimit (see Theorem~\ref{fib_cofib_thmg} and Theorem~\ref{htpy_cofinal_thmg} below). 

\begin{thm} \cite[Theorem 19.4.2]{hir03}  \label{fib_cofib_thmg}
Let $\mcalm$ be a  model category, and let $\mcalc$ be a small category. Let $\eta \colon F \lra G$ be a map of $\mcalc$-diagrams in $\mcalm$. 
\begin{enumerate}
\item[(i)] If for every object $c$ of $\mcalc$ the component $\eta[c] \colon F(c) \lra G(c)$ is a weak equivalence of cofibrant objects, then the induced map of homotopy colimits $\text{hocolim} \; F \lra \text{hocolim} \; G$ is a weak equivalence of cofibrant objects of $\mcalm$.
\item[(ii)] If for every object $c$ of $\mcalc$ the component $\eta[c] \colon F(c) \lra G(c)$ is a weak equivalence of fibrant objects, then the induced map of homotopy limits $\text{holim} \; F \lra \text{holim} \; G$ is a weak equivalence of fibrant objects of $\mcalm$.
\end{enumerate}
\end{thm}

\begin{thm}\cite[Theorem 19.6.7]{hir03}  \label{htpy_cofinal_thmg}
 Let $\mcalm$ be a  model category.  If $\theta \colon \mcalc \lra \mcald$ is homotopy left cofinal (respectively homotopy right cofinal) (see Definition~\ref{cofinal_defn}), then for every objectwise fibrant covariant (respectively contravariant) functor $F \colon \mcald \lra \mcalm$, the natural map $[\theta; F]$ from Proposition~\ref{induced_holim_propg} is a weak equivalence. 
\end{thm}

We will also need the following two propositions. The first is a generalization of Proposition~\ref{induced_holim_prop}, while the second is a generalization of Proposition~\ref{fsrik_prop}. 

\begin{prop}\cite[Proposition 19.1.8]{hir03}  \label{induced_holim_propg}
Let $\mcalm$ be a model category, and  let $\theta \colon \mcalc \lra \mcald$ be a functor between two small categories.  If $F \colon \mcald \lra \mcalm$ is an $\mcald$-diagram, then there is a canonical map 
\begin{eqnarray} \label{thetaxg}
[\theta; F] \colon \underset{\mcald}{\text{holim}}\; F \lra \underset{\mcalc}{\text{holim}}\; \theta^*F.
\end{eqnarray}
(see (\ref{theta_starx})) Furthermore, this map is natural in both variables $\theta$ and $F$ as in Proposition~\ref{induced_holim_prop}. 
\end{prop}

\begin{prop} \label{fsrik_propg} 
Let $\mcalb$ and $\bkm$ as in Definition~\ref{fsb_defn}. Let $\mcalm$ be a  model category, and 
let $F \colon \bkm \lra \mcalm$ be an objectwise fibrant cofunctor. 
\begin{enumerate}
\item[(i)] There is a natural transformation $\eta$ from $F$ to  the restriction $\fsb | \bkm$ (see Definition~\ref{fsb_defn}), which is an objectwise weak equivalence.  
\item[(ii)] If $F$ is  an isotopy cofunctor (see Definition~\ref{isotopy_cof_defn}), then so is the restriction of $\fsb$ to $\bkm$. 
\end{enumerate} 
\end{prop}

\begin{proof}
This is very similar to the proof of Proposition~\ref{fsrik_prop}. 
\end{proof}

\subsection{The category $\dv$}  \label{dv_subsection}

Consider the following data:
\begin{enumerate}
\item[$\bullet$] $U, U' \in \om$ such that $U \subseteq U'$;
\item[$\bullet$] $V \in \ok(U)$;
\item[$\bullet$] $L \colon U \times I \lra U', (x, t) \mapsto L_t(x):= L(x, t)$ is an isotopy from $U$ to $U'$ (see Definition~\ref{iso_eq_defn}).  
\end{enumerate}

The aim of this subsection is to define an important  category $\dv$ (see Definition~\ref{dv_defn}) out of these data. The definition of $\dv$ is rather technical. Roughly speaking, an object of $\dv$ is a zigzag $x$ of isotopy equivalences between $W$ and $L_1(W)$, where $W$ is a object of $\ok(V)$. Morphisms of $\dv$ are inclusions. There are two types of morphisms from $x$ to $y$ depending on the fact that $x$ and $y$ have the same length or not.  If $x$ and $y$ have the same length, a morphism from $x$ to $y$ is just an inclusion. Otherwise a morphism is still an inclusion, but more subtle.  We also prove Proposition~\ref{dv_contractible_prop}, which says  that $\dv$  is contractible.

Let us begin with the following notation, and some technical definition.  

\begin{nota} \label{ssie_not}
Given two objects $W, T \in \om$, we use the notation $W \ssie T$ to mean that $W$ is a subset of $T$ and the inclusion map $W \hra T$ is an isotopy equivalence. 
\end{nota}

In \cite[Section 3.2]{paul_don17} we introduced the concept of \textit{admissible family} $x=\{a_0, \cdots, a_{n+1}\}$ with respect to $L$ and a compact subset $K \subseteq U$. If one has different compact for each interval $[a_i, a_{i+1}]$, the family $x$ is said to be \textit{piecewise admissible}. More precisely, we have the following definition.  

\begin{defn} \label{pw_adm_defn}
Let $W \in \ok(U)$, and let $[a, b] \subseteq [0, 1]$. Let $x=\{a_0, \cdots, a_{n+1}\} \subseteq [0, 1]$ be a family such that $a_0 = a, a_{n+1} = b,$ and $a_i \leq a_{i+1}$ for all $i$. Let  $K = \{K_0, \cdots, K_n\}$ be a family of nonempty compact subsets $K_i \subseteq L_{a_i}(W)$ such that $\pi_0(K_i \hra L_{a_i}(W))$ is surjective. The family $x$ is said to be \emph{piecewise admissible} with respect to $\{K, L\colon W \times I \lra U'\}$ (or just \emph{piecewise admissible}) if for every $i$ there exists an object $W_{i(i+1)}$ of $\okm$ such that for all $s \in [a_i, a_{i+1}]$, 
\begin{eqnarray}
L_s(L_{a_i}^{-1}(K_i)) \subseteq W_{i(i+1)} \ssie L_s(W),
\end{eqnarray}
and
\begin{eqnarray} \label{vii_eqn}
\overline{W_{i(i+1)}} \subseteq L_s(W),
\end{eqnarray}
where $L_{a_i} \colon W \lra L_{a_i}(W)$ is the canonical homeomorphism induced by $L$, and $\overline{W_{i(i+1)}}$ stands for the closure of $W_{i(i+1)}$. 
\end{defn}

%\begin{enumerate}
%\item[(i)] Let $[a, b] \subseteq [0, 1]$, and let $K \subseteq L_a(W)$ be a nonempty compact subset such that $\pi_0(K \hra L_a(W))$ is surjective. A family $\{a_0, \cdots, a_{n+1}\} \subseteq I$ such that $a_0 = a, a_{n+1} = b$, and $a_i \leq a_{i+1}$ for all $i$, is called \emph{admissible} in $[a, b]$ with respect to $\{K, L\colon W \times I \lra U'\}$ if for every $i \in \{0, \cdots, n\}$ there exists $W_{i(i+1)} \in \okm$ such that for all $s \in [a_i, a_{i+1}]$, one has
%\[
%L_s(L_a^{-1}(K)) \subseteq W_{i(i+1)} \ssie L_s(W) \quad \text{and} \quad \overline{W_{i(i+1)}} \subseteq L_s(W),
%\] 
%where $L_{a} \colon W \lra L_{a}(W)$ is the canonical diffeomorphism induced by $L$, and $\overline{W_{i(i+1)}}$ stands for the closure of $W_{i(i+1)}$.  
%\item[(ii)] When $n=0$, $\{a_0, \cdots, a_{n+1}\}$ is said to be a \emph{reduced admissible family} (in the interval $[a, b]$). 
%\end{enumerate}
%\end{defn}

%Notice that Definition~\ref{adm_interval_defn}~-(i) coincides with \cite[Definition 3.9]{paul_don17} when $a=0$ and $b=1$. In the following proposition, we still denote by $L$ the restriction of $L$ to $W \times I$.
The following proposition, which will be used in the proof of Proposition~\ref{dv_contractible_prop}, can be deduced easily from  \cite[Proposition 3.10 ]{paul_don17}.
\begin{prop} \label{existence_adm_prop}
%\begin{enumerate}
%\item[(i)] Given $W, [a, b], $ and $K$ as in Definition~\ref{adm_interval_defn}, there exists an admissible family in $[a, b]$ with respect to $\{K, L\}$. 
Let $[a, b] \subseteq I,$ and let $t \in [a, b]$. Let $W \in \ok(U)$, and let $K \subseteq L_t(W)$ be a nonempty compact subset such that $\pi_0(K \hra L_a(W))$ is surjective.
\begin{enumerate}
	\item[(i)] If $t = a$ (respectively $t=b$), there exists $t' > t$ (respectively $t'' < t$) such that the family $\{t, t'\}$ (respectively $\{t'', t\}$) is admissible with respect to $\{K, L\}$. 
	\item[(ii)]  If $t \in (a, b)$, there exists $\epsilon_t >0$ such that the family $\{t-\epsilon_t, t+\epsilon_t\}$ is admissible with respect to $\{K, L\}$.   
\end{enumerate}
%there exists $a' \in I$ (respectively $a'' \in I$) such that $\{a, a'\}$ (respectively $\{a'', a\}$)  is piecewise admissible  with respect to $\{K, L\}$. 
%\end{enumerate}
\end{prop}

\begin{defn} \label{mcale_defn}
Define $\mcale$ to be the category whose objects are finite subsets $A= \{a_0, \cdots, a_{n+1}\}$ of the interval $[0, 1]$ such that $a_0 = 0, a_{n+1} = 1$ and $a_i \leq a_{i+1}$ for all $i$. Morphisms of $\mcale$ are inclusions. 
\end{defn}

\begin{defn} \label{ia_defn}
 Let $A = \{a_0, \cdots, a_{n+1}\}$ be an object of $\mcale$. Define  $\mcali_A$ to be the poset  whose objects are 
\[
\{a_0\}, \{a_1\} \cdots, \{a_n\}, \{a_0, a_1\}, \{a_1, a_2\}, \cdots, \{a_{n-1}, a_n\},
\]
and whose morphisms are inclusions $\{a_i\} \lra \{a_i, a_{i+1}\}$ and $\{a_{i+1}\} \lra \{a_i, a_{i+1}\}, 0 \leq i \leq n$. 
\end{defn}
The category $\mcali_A$ looks like a zigzag starting at $\{a_0\} = \{0\}$ and ending at $\{a_{n+1}\} =\{1\}$. For instance, if $n=2$, then 
\[
\mcali_A = \left\{\xymatrix{\{a_0\} \ar[r] & \{a_0, a_1\} & \{a_1\} \ar[l] \ar[r] & \{a_1, a_2\} & \{a_2\} \ar[l] } \right\}.
\]
%The following proposition says that the construction that sends $A$ to $\mcali_A$ is functorial. 

\begin{prop} \label{theta_ab_prop}
The construction that sends $A$ to $\mcali_A$ is a contravariant functor $\mcale \lra \text{Cat}$ from $\mcale$ to the category Cat of small categories.  
\end{prop}

\begin{proof}
Given $A, B \in \mcale$ such that $A \subseteq B$ with $B = \{b_0, \cdots, b_{m+1}\}$, we need to define a morphism
\begin{eqnarray} \label{theta_ab}
\theta_{AB} \colon \mcali_B \lra \mcali_A.
\end{eqnarray} 
Let us begin with an example. Take $A = \{a_0, a_1, a_2\}$ and $B = \{b_0, b_1, b_2, b_3\}$ such that $b_2 = a_1$ as shown Figure~\ref{ab_fig}. 
\begin{figure}[ht!]
\centering
\includegraphics[scale = 0.7]{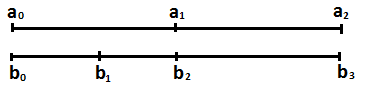}
\caption{An example of $A \subseteq B$} \label{ab_fig}
\end{figure}
The idea of the definition of $\theta_{AB}$ is as follows. First consider  the elements of $A \cap B = \{b_0, b_2, b_3\}$, and define $\theta_{AB}(\{b_0\}) = \{a_0\}, \theta_{AB}(\{b_2\}) = \{a_1\}$, and $\theta_{AB}(\{b_3\}) = \{a_2\}$. Next consider $B \backslash A = \{b_1\}$. Since $D =[a_0, a_1]$ is the smallest closed interval containing $b_1$ such that $\text{Inf} D \in A, \text{Sup} D \in A$, and $D \cap A = \{a_0, a_1\}$, we have $\theta_{AB}(\{b_1\}) := \{a_0, a_1\}$. A similar observation gives 
$
\theta_{AB}(\{b_0, b_1\}) := \{a_0, a_1\}, \theta_{AB}(\{b_1, b_2\}):= \{a_0, a_1\}, \text{ and } \theta_{AB}(\{b_2, b_3\}) = \{a_1, a_2\}. 
$ The following diagram summarizes the definition of $\theta_{AB}$. 
\[
\xymatrix{\{a_0\}  \ar[rr] &  & \{a_0, a_1\} & & \{a_1\} \ar[ll] \ar[r]  & \{a_1, a_2\} & \{a_2\} \ar[l]   \\
          \{b_0\} \ar[r] \ar[u] & \{b_0, b_1\} \ar[ru] & \{b_1\} \ar[u] \ar[l] \ar[r] & \{b_1, b_2\} \ar[lu] & \{b_2\} \ar[l] \ar[r] \ar[u] & \{b_2, b_3\} \ar[u] & \{b_3\}. \ar[l] \ar[u]}
\]

Now we give a precise definition of $\theta_{AB}$. 
%For $b \in A \cap B$, since $A \subseteq B$, there is a unique $r(b) \in \{0, \cdots, n\}$ such that $b = a_{r(b)}$. This defines a map $r \colon A \cap B \lra \{0, \cdots, n\}$. 
For $b \in B$, define 
	\[
	c(b) := \text{max}\{x \in A | \ x \leq b\} \quad \text{and} \quad d(b) := \text{min}\{x \in A| \ x \geq b\}.
	\]
Now define $\theta_{AB}$ as 
\[
\theta_{AB}(\{b\}) = \left\{ \begin{array}{ccc}
                       \{b\} & \text{if}  & b \in A \\
											 \{c(b), d(b)\} & \text{if} & b \notin A,
                      \end{array}  \right.
\]	
and 
\[
\theta_{AB}(\{b_i, b_{i+1}\}) = \{c(b_i), d(b_{i+1})\}.
\]
On morphisms of $\mcali_B$, $\theta_{AB}$ is defined in the most obvious way. Regarding the composition, if $A, B, C \in \mcale$ such  that $A \subseteq B \subseteq C$, then one obviously has 
\begin{eqnarray} \label{theta_abc}
\theta_{AC} = \theta_{AB}\theta_{BC}, 
\end{eqnarray}
which completes the proof. 
\end{proof}

We are now ready to define  $\dv$. 

\begin{defn} \label{dv_defn}
Recall  the posets $\mcale$ and $\mcali_A$ from Definition~\ref{mcale_defn} and Definition~\ref{ia_defn} respectively. Also recall the isotopy $L$ from the beginning of this subsection. The category $\dv$ is defined as follows. 
\begin{enumerate}
\item[$\bullet$] An object is a triple $(W, A, \xa)$ (or just a pair $(W, \xa$)) where $W \in \ok(V)$, $A = \{a_0, \cdots, a_{n+1}\} \in \mcale$, and $\xa \colon \mcali_A \lra \okm$ is a contravariant functor that satisfies the following three conditions: 
   \begin{enumerate}
	  \item[(a)] $\xa(\{a\}) = L_a(W)$ for all $a \in A$.
		\item[(b)] For every $i \in \{0, \cdots, n\}$, for every $s \in [a_i, a_{i+1}]$,
		 \[
		 \xa(\{a_i, a_{i+1}\}) \ssie L_s(W).
		 \] 
		(See Notation~\ref{ssie_not} for the meaning of \lq\lq$\ssie$\rq\rq{}.)
	\item[(c)] 	 For every $i \in \{0, \cdots, n\}$, for every $s \in [a_i, a_{i+1}]$,
		 \begin{eqnarray} \label{xai_eqn}
		 \overline{\xa(\{a_i, a_{i+1}\})} \subseteq L_s(W).
		 \end{eqnarray}
	\end{enumerate}
\item[$\bullet$] A morphism from $(W, A, \xa)$ to $(T, B, \yb)$ consists of a triple $(f, g, \Lambda_{AB})$ (or just $\Lambda_{AB}$) where $f \colon W \hra T$ and $g \colon A \hra B$ are both the inclusion maps, and $\Lambda_{AB} \colon \xa\theta_{AB} \lra \yb$ is a natural transformation.  
\end{enumerate}
\end{defn}

In other words, an object of $\dv$ is a zigzag of isotopy equivalences between $L_0(W) = W$ and  $L_1(W)$, where $W \in \ok(V)$. For instance, when $A = \{a_0, a_1, a_2\}$, an object looks like (\ref{xzig}). 
\begin{eqnarray} \label{xzig}
(W,\xa) = \left\{ \xymatrix{W=X_0 & X_{01} \ar[l]_-{\simeq} \ar[r]^-{\simeq} & X_1 & X_{12} \ar[l]_-{\simeq} \ar[r]^-{\simeq} & X_2 = L_1(W)}  \right\}.
\end{eqnarray}
There are two kind of morphisms from $(W,\xa)$ to $(T, \yb)$ depending on the fact that $A =B$ or $A$ is a proper subset of $B$. These morphisms are illustrated by (\ref{mor1}) and (\ref{mor2}). 
\begin{eqnarray} \label{mor1}
\xymatrix{X_0 \ar[d] & X_{01} \ar[l]_-{\simeq} \ar[r]^-{\simeq} \ar[d] & X_1 \ar[d] & X_{12} \ar[l]_-{\simeq} \ar[r]^-{\simeq} \ar[d] & X_2 \ar[d]  \\
Y_0 & Y_{01} \ar[l]^-{\simeq} \ar[r]_-{\simeq} & Y_1 & Y_{12} \ar[l]^-{\simeq} \ar[r]_-{\simeq} & Y_2. }
\end{eqnarray}

\begin{eqnarray} \label{mor2}
\xymatrix{X_0 \ar[d] & & X_{01} \ar[ll]_-{\simeq} \ar[rr]^-{\simeq} \ar[d] \ar[ld] \ar[rd] &  & X_1 \ar[d] &  X_{12} \ar[l]_-{\simeq} \ar[r]^-{\simeq} \ar[d]  & X_2 \ar[d]  \\
Y_0 & Y_{01} \ar[l]^-{\simeq} \ar[r]_-{\simeq} & Y_1 & Y_{12} \ar[l]^-{\simeq} \ar[r]_-{\simeq} & Y_2 & Y_{23} \ar[l]^-{\simeq} \ar[r]_-{\simeq} & Y_3.}
\end{eqnarray}

\begin{rmk} \label{associated_rmk}
To any piecewise admissible family $A$ (see Definition~\ref{pw_adm_defn}), one can associate a canonical object $(W, \xa)$ of $\dv$ by letting $\xa(\{a_i, a_{i+1}\}) := W_{i(i+1)}.$ 
\end{rmk}       

\begin{prop} \label{dv_contractible_prop} 
The category $\dv$ is contractible.               
\end{prop}      

\begin{proof}
It suffices to show that $\dv$ is \textit{filtered}, that is, it satisfies the following two conditions: 
\begin{enumerate}
\item[(1)] For every pair of objects $(W,\xa)$ and $(T, \yb)$ there are morphisms to a common object $(W, \xa) \lra (S, \zc)$ and $(T, \yb) \lra (S, \zc)$;
\item[(2)] For every pair of parallel morphisms $\Lambda_{AB}, \Lambda'_{AB} \colon (W, \xa) \lra (T, \yb)$, there is some morphism $\Lambda_{BC} \colon (T, \yb) \lra (S, \zc)$ such that $\Lambda_{BC}\Lambda_{AB} = \Lambda_{BC} \Lambda'_{AB}$. 
\end{enumerate}
Since $\dv$ is a poset by definition, it clearly satisfies (2). To check (1), let $(W, \xa), (T, \yb) \in \dv$. Set $D = A \cup B$. Certainly $D$ is a finite subset, denoted $\{d_0, \cdots, d_{p+1}\}$, of $I$ such that $d_0 = 0, d_{p+1} =1$, and $d_i \leq d_{i+1}$ for all $i$.  One can write the intervals $[d_i, d_{i+1}], 0 \leq i \leq p,$ as 
\[
[d_i, d_{i+1}] = \big[a_{r(i)}, a_{r(i)+1}\big] \cap \big[b_{s(i)}, b_{s(i)+1}\big],
\]
where 
\[
a_{r(i)} := \text{max}\left\{x \in A| \ x \leq d_i \right\} \quad \text{and} \quad b_{s(i)} := \text{max}\left\{y \in B| \ y \leq d_i\right\}.
\]
Of course, $a_{r(i)+1}$ (respectively $b_{s(i)+1}$) is the successor of $a_{r(i)}$ in $A$ (respectively the successor of $b_{s(i)}$ in $B$). Note that the interior of $[d_i, d_{i+1}]$ does not intersect either $A$ or $B$. 

Take $S = V$. The idea of the construction of $\zc$ is to subdivide each $[d_i, d_{i+1}]$ into small intervals $[a, b]$ such that there exists $Z_{ab} \in \okm$ that is contained in $L_a(V) \cap L_b(V)$ and that contains both $\xa\left(\{a_{r(i)}, a_{r(i)+1}\}\right)$ and $\yb\left(\{b_{s(i)}, b_{s(i)+1}\}\right)$. So let $i \in \{0, \cdots, p\}$, and let $t \in [d_i, d_{i+1}]$.  Thanks to (\ref{xai_eqn}) one can  consider the compact subset $\mcalk_i \subseteq L_t(V)$ defined as 
\[
\mcalk_i = \overline{\xa\left(\left\{a_{r(i)}, a_{r(i)+1} \right\} \right)} \bigcup \overline{\yb\left(\left\{b_{s(i)}, b_{s(i)+1} \right\} \right)}.
\]
If $t \in (d_i, d_{i+1})$ then by  Proposition~\ref{existence_adm_prop} there exist $\epsilon_t > 0$ and $Z_{i(i+1)} \in \okm$ such that for all $u \in [t-\epsilon_t, t+\epsilon_t]$, 
\[
L_u(L_t^{-1}(\mcalk_i)) \subseteq Z_{i(i+1)} \ssie L_u(V) \quad \text{and} \quad \overline{Z_{i(i+1)}} \subseteq L_u(V).  
\]
%where $L_t \colon V \lra L_t(V)$ is the canonical diffeomorphism induced by $L$. 
%Define $\epsilon_1 := \text{min}(t', d_{i+1}) - t$.  Then, it is clear that 
Clearly one has
\[
\xa\left(\{a_{r(i)}, a_{r(i)+1}\}\right) \subseteq  L_{t-\epsilon_t}(V) \cap Z_{i(i+1)} \cap L_{t+\epsilon} (V)
\]
and
\[
\yb\left(\{b_{s(i)}, b_{s(i)+1}\}\right) \subseteq  L_{t-\epsilon_t}(V) \cap Z_{i(i+1)} \cap L_{t+\epsilon} (V). 
\] 
If $t= d_i$ (respectively $t=d_{i+1}$) there is an admissible family $\{d_i, t'\}$ (respectively $\{t'', d_{i+1}\}$) again by  Proposition~\ref{existence_adm_prop}. 
%Similarly, there exist  $Z'_{i(i+1)} \in \okm$ and $\epsilon_2 > 0$ such that both $\xa\left(\{a_{r(i)}, a_{r(i)+1}\}\right)$ and $\yb\left(\{b_{s(i)}, b_{s(i)+1}\}\right)$ are contained in the intersection $L_{t-\epsilon_2}(V) \cap Z'_{i(i+1)} \cap L_t(V)$. Taking $\epsilon_t:= \text{min}(\epsilon_1, \epsilon_2)$, and 
Letting $t$ vary in $[d_i, d_{i+1}]$ one obtains an open cover, $[d_i, t') \cup \{(t-\epsilon_t, t+\epsilon_t)\}_t \cup (t'', d_{i+1}])$, of $[d_i, d_{i+1}]$. Now, applying the compactness we get an ordered finite subset $C^i = \{c^i_0, \cdots, c^i_{n_i}\}$ of $[d_i, d_{i+1}]$ such that $c^i_0 = d_i$, $c^i_{n_i} = d_{i+1}$, and for every $j$ the interval $[c^i_j, c^i_{j+1}]$ is contained in one of the open subsets from the cover. This implies that $C:= \cup_{i=0}^p C^i$ is piecewise admissible
% (with respect to $L \colon V \times I \lra U'$) 
and contains both $A$ and $B$.  Moreover,  it is clear that the  associated object $(V, \zc)$ of $\dv$ (as in Remark~\ref{associated_rmk}) has  the desired property.
% by (\ref{xa_eqn}), (\ref{yb_eqn}), and the fact that $\xa$ and $\yb$ are both objects of $\dv$. 
This ends the proof.   
\end{proof}

\begin{coro} \label{theta_zo_lem}
The functors
\[
\theta_0 \colon \dv \lra \ok(V) \quad \text{and} \quad \theta_1 \colon \dv \lra \ok(L_1(V))
\]
defined as $\theta_0(W, \xa) = W$ and $\theta_1(W, \xa) = L_1(W)$ are both homotopy right cofinal (see Definition~\ref{cofinal_defn}). 
\end{coro}

\begin{proof}
For every $W \in \ok(V)$ the under category $W \downarrow \theta_0$ is contractible. This works exactly as the proof of Proposition~\ref{dv_contractible_prop}. Similarly, one can show that $\theta_1$ is homotopy right cofinal. 
\end{proof}

\begin{prop} \label{dv_functor_prop}
The construction $\mcald \colon \ok(U) \lra \text{Cat}$ that sends $V$ to $\dv$ is a covariant functor. 
\end{prop}

\begin{proof} It is very easy to establish. For a morphism $V \hra V'$ of $\ok(U)$, we define $\theta \colon \dv \lra \dvp$ as   $\theta(W, \xa) = (W, \xa)$. Certainly  this defines a functor from $\dv$ to $\dvp$.    
\end{proof}

%\begin{rmk} \label{xa_xap_rmk}
%\begin{enumerate}
%\item[(i)] By construction, there is a canonical natural transformation
%\begin{eqnarray}
%\alpha \colon \xa \lra \mathcal{X}'_A. 
%\end{eqnarray}
%\item[(ii)]Notice that (\ref{lslai_eqn}) might be not true if the inclusion $V \hra V'$ is not anymore an isotopy equivalence. So $\mcald$ might be not a canonical functor when replacing its domain $\okt(U)$ by $\ok(U)$. 
%\end{enumerate}
%\end{rmk}

\subsection{The functors $H, P_0, P_1 \colon \dv \lra \mcalm$}   \label{hp_subsection}

The goal of this subsection is to define three important functors, $H, P_0, P_1 \colon \dv \lra \mcalm$, and two  natural weak equivalences, $\eta_0 \colon H \lra P_0$ and $\eta_1 \colon H \lra P_1$. We  will  use them in the proof of Theorem~\ref{iso_cof_thm}, which will be done at Subsection~\ref{iso_cof_subsection}. In this subsection $\mcalm$ is a  model category, $F \colon \okm \lra \mcalm$ is an isotopy cofunctor (see Definition~\ref{isotopy_cof_defn}), and $F^{!} \colon \om \lra \mcalm$ is the cofunctor defined by (\ref{fsrik_defn}). We continue to use  the same data as those  provided at the beginning of Subsection~\ref{dv_subsection}.

\subsubsection{The functor $H \colon \dv \lra \mcalm$}

Before we define $H \colon \dv \lra \mcalm$, we need to first recall a certain model category of diagrams, and next define $\Psi \colon \dv \times \mcale \lra \mcalm$ (for the categories $\dv$ and $\mcale$, see Definition~\ref{dv_defn} and Definition~\ref{mcale_defn} respectively), which is functorial in each argument.

For $E \in \mcale$, consider the category $\mcalm^{\mcali_E}$ of $\mcali_E$-diagrams in $\mcalm$ (recall the poset $\mcali_E$ from Definition~\ref{ia_defn}). In the literature there exist many model structures on  $\mcalm^{\mcali_E}$. But for our purposes we endow it with the one described by Dwyer and Spalinski in \cite[Section 10]{dwyer_spa95}. First recall that this model structure is only defined for diagrams indexed by \textit{very small categories} (see the paragraph just after 10.13 from \cite{dwyer_spa95}), which is the case for $\mcali_E$. Next recall that this model structure states that weak equivalences and cofibrations are both objectwise. A map $\mcalx \lra \mcaly$ is a fibration if certain explicit morphisms in $\mcalm$ associated with $\mcalx \lra \mcaly$ are  fibrations. (See for example (10.9), (10.10), and Proposition 10.11 from \cite{dwyer_spa95}.) One of the advantages of this model structure is the fact that any diagram admits an explicit fibrant replacement as shown the following illustration. 
\begin{expl} \label{fib_repl_expl}
Consider the following objectwise fibrant diagram in $\mcalm$. 
\[
\mcalx = \left\{ \xymatrix{X_0 \ar[r]^-{f_0} & X_{01} & X_1 \ar[l]_-{f_1} \ar[r]^-{f_2}  & X_{12} & X_2  \ar[l]_-{f_3}} \right\}
\] 
Then its fibrant replacement, $R\mcalx$, is the second row of the following commutative diagram
\[
\xymatrix{X_0 \ar[r]^-{f_0} \ar@{>->}[d]^-{\sim}_-{g_0} & X_{01} \ar@{>->}[d]^-{\sim}_-{id} & X_1 \ar[l]_-{f_1} \ar[r]^-{f_2} \ar@{>->}[d]^-{\sim}_-{g_1} & X_{12} \ar@{>->}[d]^-{\sim}_-{id} & X_2 \ar[l]_-{f_3} \ar@{>->}[d]^-{\sim}_-{g_2}  \\
\xtild_{0} \ar@{->>}[r] & X_{01} & \xtild_1 \ar@{->>}[l] \ar@{->>}[r] & X_{12} & \xtild_2. \ar@{->>}[l]  }
\]
To get $R\mcalx$, first  we take a fibrant replacement $\xtild_{i(i+1)}, 0 \leq i \leq n$ (here $n=1$), of $X_{i(i+1)}$ in $\mcalm$. Since $\mcalx$ is objectwise fibrant, we then take $\xtild_{i(i+1)} = X_{i(i+1)}$. Next the functorial factorization of the composition $idf_0$ (respectively  $idf_3$) provides $\xtild_0$ (respectively $\xtild_{n+1} = \xtild_2$).  Lastly, $\xtild_1$ comes from the functorial factorization 
\[
\xymatrix{X_1 \ar[rr]^-{(idf_1, idf_2)}  \ar@{>->}[rd]^-{\sim}_-{g_1}  &  &  X_{01} \times X_{12} \\
      &   \widetilde{X}_1 \ar@{->>}[ru] }
\]
\end{expl}

\begin{rmk} \label{trx_rmk}
Let $\theta \colon \mcali \lra \mcalj$ be a functor between small categories, and let $\mcalx \colon \mcalj \lra \mcalm$  be an $\mcalj$-diagram in $\mcalm$. Then $\theta^*(R\mcalx)$ is not equal   to $R \theta^*(\mcalx)$ in general, but there is a natural map $\theta^*(R\mcalx) \lra R \theta^*(\mcalx)$. This map comes directly from the way we construct our fibrant replacements. 
\end{rmk}

Now we define $\Psi \colon \dv \times \mcale \lra \mcalm$. First recall the covariant functor $\theta_{AB} \colon \mcali_B \lra \mcali_A$
 defined in the course of the proof of Proposition~\ref{theta_ab_prop}. For $((W,\xa), E) \in \dv \times \mcale$ such that  $A \subseteq E$,  one can consider the composition 
\begin{eqnarray} \label{ieae_comp}
\xymatrix{\mcali_E \ar[rr]^-{\theta_{AE}} & & \mcali_A \ar[rr]^-{\xa} & & \okm \ar[rr]^-{F^{!}} & & \mcalm}, 
\end{eqnarray} 
which is nothing but an $\mcali_E$-diagram in $\mcalm$. Define  $\Psi((W,\xa), E) \in \mcalm$ as 
\begin{eqnarray} \label{psi_defn}
\Psi((W,\xa), E) = \left\{ \begin{array}{ccc}
											 \underset{\mcali_E}{\text{lim}} \; RF^{!} \theta_{AE}^*(\xa) & \text{if} & A \subseteq E \\
											 \emptyset & \text{if} & \text{$A$ is not contained in $E$},
                       \end{array} \right.
\end{eqnarray} 
where  $\emptyset$ stands for the initial object of $\mcalm$. 

\begin{prop}
The construction $\Psi \colon \dv \times \mcale \lra \mcalm$ that sends $((W,\xa), E)$ to $\Psi((W,\xa), E)$ is 
%a bifunctor, which is 
contravariant in the first variable and covariant in the second one. 
\end{prop}

\begin{proof}
Let $((W,\xa), E) \in \dv \times \mcale$. We have to prove two things. 
\begin{enumerate}
\item[$\bullet$] Functoriality in the first variable. Let $(T,\yb) \in \dv$ such that $A \subseteq B$, and let $\Lambda_{AB}$ be a morphism in $\dv$ from $(W,\xa)$ to $(T,\yb)$. Then, by Definition~\ref{dv_defn}, $\Lambda_{AB} \colon \xa \theta_{AB} \lra \yb$ is a natural transformation. If $B \subseteq E$, then one has the composition $F^{!}\Lambda_{AB} \theta_{BE} \colon F^{!}\xa\theta_{AB}\theta_{BE} \lla F^{!}\yb\theta_{BE}$ (remember that $F^{!}$ is contravariant, and that $\theta_{BE}$ is covariant), which is the same as 
\begin{eqnarray} \label{fab}
F^{!}\Lambda_{AB} \theta_{BE} \colon F^{!}\theta_{AE}^*(\xa) \lla F^{!}\theta_{BE}^*(\yb)
\end{eqnarray}
since $\theta_{AB}\theta_{BE} = \theta_{AE}$ by (\ref{theta_abc}). This induces a morphism
\[
\Psi(\Lambda_{AB}, id) := \text{lim}(RF^{!}\Lambda_{AB}\theta_{BE}) \colon \Psi((W,\xa), E) \lla \Psi((T,\yb), E). 
\]
If $E$ does not contain $B$, then $\Psi((T,\yb), E)$ is the initial object by definition, and therefore $\Psi(\Lambda_{AB}, id)$ is the unique morphism from $\emptyset$ to $\Psi((W,\xa), E)$. 

\item[$\bullet$] Functoriality in the second variable. Let $E' \in \mcale$ such that $E \subseteq E'$. If $A \subseteq E$, then we have
\[
\theta^*_{EE'}\left(F^{!}\theta^*_{AE}(\xa) \right) = F^{!}\theta^*_{AE'}(\xa)
\]
by (\ref{theta_abc}) and (\ref{theta_starx}). The map $$\Psi(id, E \hra E') \colon \Psi((W,\xa), E) \lra \Psi((W,\xa), E')$$ is then defined to be the composition
\[
\underset{\mcali_E}{\text{lim}} \; RF^{!} \theta_{AE}^* (\xa) \lra \underset{\mcali_{E'}}{\text{lim}} \; \theta^*_{EE'} \left(RF^{!} \theta_{AE}^* (\xa)\right) \lra \underset{\mcali_{E'}}{\text{lim}} \; R\theta^*_{EE'} \left(F^{!} \theta_{AE}^* (\xa)\right). 
\]
Here the first arrow is the canonical map induced by $\theta_{EE'} \colon \mcali_{E'} \lra \mcali_E$, and the second is the natural map that comes directly from the way fibrant replacements of $\mcali_{E'}$-diagram are constructed (see Example~\ref{fib_repl_expl} and Remark~\ref{trx_rmk}). As before, the case where $A$ is not contained in $E$ is obvious. 

%\item[$\bullet$] Compatibility. It follows directly from the definitions. 

%Let $\Lambda_{AB}$ be a morphism from $\xa$ to $\yb$ in $\dv$, and let $E \subseteq E'$ be a morphism in $\mcale$. Assume $B \subseteq E$. Then, since $\theta^*_{EE'}(F^{!}\Lambda_{AB}\theta_{BE}) = F^{!}\Lambda_{AB}\theta_{BE'}$, and since there is a morphism $F^{!}\Lambda_{AB} \theta_{BE} \colon F^{!}\theta_{BE}^*(\yb) \lra F^{!}\theta_{AE}^*(\xa)$ given by (\ref{fab}),  it follows that the square 
%\[
%\xymatrix{\underset{\mcali_E}{\text{holim}}\; F^{!}\theta_{BE}^*(\yb) \ar[rrr]^-{[\theta_{EE'}; F^{!}\theta_{BE}^*(\yb)]} \ar[d]_-{\text{holim}(F^{!}\Lambda_{AB}\theta_{BE})}  &  &  &  \underset{\mcali_{E'}}{\text{holim}}\; F^{!}\theta_{BE'}^*(\yb) \ar[d]^-{\text{holim}(F^{!}\Lambda_{AB}\theta_{BE'})}  \\
% \underset{\mcali_E}{\text{holim}}\; F^{!}\theta_{AE}^*(\xa) \ar[rrr]_-{[\theta_{EE'}; F^{!}\theta_{AE}^*(\xa)]} &  &  &  \underset{\mcali_{E'}}{\text{holim}}\; F^{!}\theta_{AE'}^*(\xa)  }
%\]
%commutes by (\ref{theta_square}). Clearly, if $B$ is not contained in $E$, the preceding square is obviously commutative. 
\end{enumerate}
This proves the proposition.  
\end{proof}

Before we define $H \colon \dv \lra \mcalm$, we need to equip the category $\mcalm^{\mcale}$ of $\mcale$-diagrams in $\mcalm$ with a nice model structure. Thanks to the fact that the category $\mcale$ is a \textit{direct category} (see \cite[Definition 5.1.1]{hovey99}), and therefore a \textit{Reedy category} (see \cite[Definition 5.2.1]{hovey99}), we can endow $\mcalm^{\mcale}$ with a Reedy model structure that we now recall.  For $E \in \mcale$, we define the \textit{latching space} functor $L_E \colon \mcalm^{\mcale} \lra \mcalm$ as follows. Let $\mcale_E$ be the category of non-identity maps in $\mcale$ with codomain $E$, and define $L_E$ to be the composite 
\[
L_E \colon \xymatrix{\mcalm^{\mcale} \ar[r] & \mcalm^{\mcale_E} \ar[rr]^-{\text{colim}} &  & \mcalm}
\]  
where the first arrow is restriction. Clearly there is a natural transformation $L_E\mcalx \lra \mcalx(E)$. 

\begin{thm}\cite[Theorem 5.1.3]{hovey99} \label{struc_me_thm}
There exists a model structure  on  the category $\mcalm^{\mcale}$ of $\mcale$-diagrams in $\mcalm$ such that weak equivalences and fibrations are objectwise. Furthermore, a map $\mcalx \lra \mcaly$ is  a (trivial) cofibration if and only if the induced map $\mcalx(E) \coprod_{L_E\mcalx} L_E\mcaly \lra \mcaly(E)$ is a (trivial) cofibration for all $E$. 
\end{thm}

Note that any object $\mcalx$  of $\mcalm^{\mcale}$ has an explicit cofibrant replacement $Q\mcalx \colon \mcale \lra \mcalm$ obtained by induction as follows. (Recall that by definition $\{0, 1 \}$ is the initial object of $\mcale$). First take the cofibrant replacement $Q\mcalx(\{0, 1\})$ of $\mcalx(\{0, 1\})$. Next, for   any other object  $E \in \mcale$, $Q\mcalx(E)$ comes from the functorial factorization of the obvious map 
\[
\underset{E' \subset E}{\text{colim}} \; Q\mcalx(E') \lra \mcalx(E), 
\]
where $E' \in \mcale$ runs over the set of proper subsets of $E$. As an example, the cofibrant replacement of $\Psi((W,\xa), -)$ is an $\mcale$-diagram on the form 
\begin{eqnarray} \label{qpsi_shape}
Q\Psi((W,\xa), -) = \xymatrix{   &   &  \bullet  \cdots \ar@{.}[d] \\
                                               \emptyset \ar[r] & Q\Psi((W,\xa), A) \ar@{.>}[ru] \ar@{.>}[r] \ar@{.>}[rd] & \bullet \cdots \ar@{.}[d]\\
                                                           &  &   \bullet \cdots  }  
\end{eqnarray} 

\begin{rmk} \label{cofibrant_rmk}
By construction, every object of the diagram $Q\mcalx$ is cofibrant in $\mcalm$. 
\end{rmk}

\begin{prop} \label{hoco_co_prop}
The natural map 
\[
\underset{\mcale}{\text{hocolim}} \; Q\Psi((W, \xa), -) \lra \underset{\mcale}{\text{colim}} \; Q\Psi((W, \xa), -)
\]
is a weak equivalence. 
\end{prop}

\begin{proof}
This follows from \cite[Theorem 19.9.1]{hir03}. 
\end{proof}

We come to the definition of $H$. 

\begin{defn}
Recall $\Psi$ from (\ref{psi_defn}), and define the functor $H \colon \dv \lra \mcalm$ as 
\begin{eqnarray} \label{functor_h}
H(W, \xa) = \underset{E \in \mcale}{\text{colim}} \; Q\Psi((W, \xa), E). 
\end{eqnarray}
\end{defn}

\begin{rmk}
Let $\mcale_A \subseteq \mcale$ denote the full subcategory whose objects are $E$ containing $A$, and  let $\psit((W,\xa), -)$ denote the restriction of $\Psi((W,\xa), -)$ to $\mcale_A$. Then one has
\begin{eqnarray} \label{hxa_formula}
H(W,\xa) = \underset{ \mcale}{\text{colim}} \; Q\Psi((W,\xa), -) \cong \underset{ \mcale_A}{\text{colim}} \; Q\psit((W,\xa), -). 
\end{eqnarray}
The   isomorphism \lq\lq $\cong$\rq\rq{}  follows from the fact that the diagram (\ref{qpsi_shape}) contains the initial object, $\emptyset$, of $\mcalm$.  
\end{rmk}

Now we define another map (the map $h$ below) which will be used in the next subsection. Recalling the data provided at the beginning of Subsection~\ref{dv_subsection}, and using definitions, we can easily see that for every $E \in \mcale_A$, for every $x \in \mcali_E$, one has $\xa \theta_{AE}(x) \subseteq U'$. Applying the contravariant functor $F^{!}$ to this latter inclusion, we get maps $F^{!}(U') \lra F^{!} \xa \theta_{AE}(x)$. This induces a natural transformation  $F^{!}(U') \lra \psit(\xa, -)$ between two $\mcale_A$-diagrams in $\mcalm$, the first one being the constant diagram (recall that $RF^{!}(U') = F^{!}(U')$ by the assumption that $F$ is objectwise fibrant and by Theorem~\ref{fib_cofib_thm}). Now taking the cofibrant  replacement of this latter map, passing to the colimit, and using (\ref{hxa_formula}) we have a map 
\begin{eqnarray} \label{map_h}
h \colon QF^{!}(U') \lra H(W,\xa), 
\end{eqnarray}
which is natural in $(W,\xa)$.

\subsubsection{The functors $P_0, P_1 \colon \dv \lra \mcalm$}

To define $P_0$ and $P_1$,  we will first define 
\[
\Phi_0 \colon \dv \times \mcale \lra \mcalm \quad \text{and} \quad \Phi_1 \colon \dv \times \mcale \lra \mcalm. 
\]
To do this, we need to introduce some notation. If $E= \{a_0, \cdots, a_{n+1}\}$ is an object of $\mcale$ and  $\mcalx \colon \mcali_E \lra \mcalm$ is a functor,  we define two objects $\phi_0 \mcalx$ and $\phi_1 \mcalx$  of $\mcalm$ as 
\[
\phi_0 \mcalx := \mcalx(\{a_0\}) \quad \text{and} \quad \phi_1 \mcalx := \mcalx(\{a_{n+1}\}).
\]
In other words, $\phi_0 \mcalx$ is the first object of the zigzag $\mcalx$, while $\phi_1 \mcalx$ is the last one. 

Let $((W,\xa), E) \in \dv \times \mcale$. If  $A \subseteq E$, then one can consider the composition $F^{!}\xa\theta_{AE} \colon \mcali_E \lra \mcalm$  (from (\ref{ieae_comp})), which is an object of $\mcalm^{\mcali_E}$. Let $RF^{!}\xa\theta_{AE}$ denote its fibrant replacement with respect to the Dwyer-Spalinski model structure we described before (see Example~\ref{fib_repl_expl} for an illustration of what we call fibrant replacement). Define $\Phi_0((W,\xa), E)$ as 
\begin{eqnarray} \label{phiz_defn}
\Phi_0((W,\xa), E) = \left\{ \begin{array}{ccc}
                         \phi_0R F^{!} \xa \theta_{AE} & \text{if} & A \subseteq E \\
													\emptyset  & \text{if} & \text{$A$ is not contained in $E$}. 
                         \end{array}  \right.
\end{eqnarray}

Replacing $\phi_0$ by $\phi_1$ in (\ref{phiz_defn}), we have the definition of $\Phi_1((W,\xa), E)$.  The following remark about $\Phi_0$ and $\Phi_1$ is important. 

\begin{rmk} \label{phixa_rmk}
By inspection, for every $((W,\xa),E) \in \dv \times \mcale$, one has 
\[
\Phi_0((W,\xa), E) = \Phi_0((W,\xa), A) \quad \text{and} \quad \Phi_1((W,\xa), E) = \Phi_1((W,\xa), A),
\]
provided that $A \subseteq E$. This easily comes from three things: the definition of $\dv$, that of $\theta_{AE}$, and the way   fibrant replacements of $\mcali_E$-diagrams in $\mcalm$ are constructed (see Example~\ref{fib_repl_expl}). 
\end{rmk}

\begin{prop}
The construction $\Phi_i \colon \dv \times \mcale \lra \mcalm, i=0, 1$, that sends $((W,\xa), E)$ to $\Phi_i((W,\xa), E)$ is 
%a bifunctor, which is 
contravariant in the first variable and covariant in the second one. %The same assertion holds for $\Phi_1$. 
\end{prop}

\begin{proof}
For the functoriality in the first variable, let $\Lambda_{AB}$  be a morphism of $\dv$ from $(W,\xa)$ to $(T,\yb)$, and consider the map $F^{!}\Lambda_{AB}\theta_{BE}$ from (\ref{fab}). Its fibrant replacement  gives 
\[
\Phi_0(\Lambda_{AB}, id) \colon \Phi_0((W,\xa), E) \lla \Phi_0((T,\yb), E).
\] The functoriality in the second variable is obvious by Remark~\ref{phixa_rmk}. In fact, if $i \colon E \hra E'$ is a morphism of $\mcale$ then $\Phi_0((W,\xa), i) = id$ when $A \subseteq E$. 
%Lastly,  the compatibility follows directly from the definitions. 
 A similar proof can be performed with $\Phi_1$ in place of $\Phi_0$. 
\end{proof}

\begin{defn} Recall $\Phi_0$ and $\Phi_1$ from (\ref{phiz_defn}), and define $P_0 \colon \dv \lra \mcalm$ and $P_1 \colon \dv \lra \mcalm$ as
\begin{eqnarray} \label{functor_po}
P_0(W,\xa) = \underset{\mcale}{\text{colim}} \; Q\Phi_0((W,\xa), -) \quad \text{and} \quad  P_1(W,\xa) = \underset{\mcale}{\text{colim}} \; Q\Phi_1((W,\xa), -),
\end{eqnarray}
where $Q\Phi_i((W, \xa), -), i \in \{0, 1\}$, is the cofibrant replacement of the $\mcale$-diagram $\Phi_i((W, \xa), -)$ with respect  to the model structure given by Theorem~\ref{struc_me_thm}. 
\end{defn}
 
 Now we define two important maps ($p_0$ and $p_1$ below) that will be also used in the next subsection. First, recalling the definition of $\dv$ (from Definition~\ref{dv_defn}) and that of $\theta_{AB}$ (from (\ref{theta_ab})), one can see that the functor $F^{!}\xa\theta_{AE}$ from (\ref{ieae_comp}) is nothing but a zigzag in $\mcalm$ starting at $F^{!}(W)$ and ending at $F^{!}(L_1(W))$.  If  $\phit_i(\xa, -), i \in \{0, 1 \}$ denotes the restriction of $\Phi_i(\xa, -)$ to $\mcale_A$, by the definition of the fibrant replacement, (\ref{phiz_defn}) immediately implies the existence of two natural weak equivalences:  
\[
F^{!}(W) \stackrel{\sim}{\lra} \phit_0((W,\xa), -) \quad \text{and} \quad F^{!} (L_1(W)) \stackrel{\sim}{\lra} \phit_1((W,\xa), -).
\]
 Of course the functors involves are viewed as $\mcale_A$-diagrams.  Taking the cofibrant replacement with respect to the model structure described in Theorem~\ref{struc_me_thm}  of those maps, we get  $QF^{!}(W) \stackrel{\sim}{\lra} Q\phit_0((W,\xa), -)$ and $QF^{!}(L_1(W)) \stackrel{\sim}{\lra} Q\phit_1((W,\xa), -)$. Passing to the colimit, and using the observation (\ref{hxa_formula}) we did for $H(W,\xa)$, we have weak equivalences
 \begin{eqnarray} \label{qfv_colim}
 p_0 \colon QF^{!}(W) \stackrel{\sim}{\lra}   P_0(W,\xa) \quad \text{and} \quad p_1 \colon QF^{!}(L_1(W)) \stackrel{\sim}{\lra}  P_1(W,\xa).  
 \end{eqnarray}
Notice that these maps are both natural in $(W,\xa)$. Also notice that by Remark~\ref{phixa_rmk} the diagram $\phit_i((W,\xa), -)$ is in fact the constant diagram whose value is $\phit_i((W,\xa), A)$.

\subsubsection{The maps $\eta_i \colon H \lra P_i, i=0, 1$}

The aim here is to show that the definitions of $H$ (\ref{functor_h}), $P_0,$ and $P_1$ (\ref{functor_po}) imply the existence of  natural weak equivalences 
\begin{eqnarray} \label{etai}
\eta_0 \colon \xymatrix{H \ar[r]^-{\sim} & P_0} \quad \text{and} \quad \eta_1 \colon \xymatrix{H \ar[r]^-{\sim} & P_1}.  
\end{eqnarray}
We will  show the existence of $\eta_0$; the existence of $\eta_1$ is similar.  The idea is to define for every $((W,\xa), E) \in \dv  \times \mcale$ a natural map 
\[
\alpha[(W,\xa), E] \colon \Psi((W,\xa), E) \lra \Phi_0((W,\xa), E),
\]  
and show that it is a weak equivalence. Applying the cofibrant replacement functor, and then the  colimit functor to $\alpha[(W,\xa), -]$, and using Remark~\ref{cofibrant_rmk},  Theorem~\ref{fib_cofib_thmg}, and Proposition~\ref{hoco_co_prop} we will then deduce that $\eta_0$ is a weak equivalence. So let $((W,\xa), E) \in \dv \times \mcale$.  If $A$ is not contained in $E$, then $\Psi((W,\xa), E) = \emptyset = \Phi_0((W, \xa), E)$ by definition, and therefore $\alpha[(W,\xa), E]$ is the obvious map. Assume $A \subseteq E$. Then  $\Psi((W,\xa), E) = \underset{\mcali_E}{\text{lim}} \; RF^{!}\theta^*_{AE}(\xa)$ by definition. Define $\alpha[(W,\xa), E]$ to be obvious map
\begin{eqnarray} \label{lim_phi}
\underset{\mcali_E}{\text{lim}} \; RF^{!} \theta^*_{AE}(\xa) \lra \phi_0RF^{!} \theta^*_{AE}(\xa) = \Phi_0((W,\xa), E).
\end{eqnarray}
This map is so obvious because  $\phi_0RF^{!} \theta^*_{AE}(\xa)$ is nothing but  one piece from the diagram $RF^{!} \theta^*_{AE}(\xa)$. It is straightforward to check the naturality of $\alpha[(W,\xa), E]$ in both variables.  One can also see that (\ref{lim_phi}) is a weak equivalence essentially by the following reason. First, since  $\xa$ is a zigzag of isotopy equivalences by the condition (b) from Definition~\ref{dv_defn}, and   since $F^{!} | \okm$ is an isotopy cofunctor by Proposition~\ref{fsrik_propg}, it follows that every morphism of the diagram $F^{!}\theta^*_{AE} (\xa)$ is a weak equivalence. This implies that every morphism of  $RF^{!}\theta^*_{AE} (\xa)$ is a weak equivalence as well.   Moreover $RF^{!}\theta^*_{AE} (\xa)$ is fibrant. So every morphism of  the diagram $RF^{!}\theta^*_{AE} (\xa)$ is a weak equivalence which is also a fibration. Thanks to the shape of this diagram, one can compute its limit by taking successive pullbacks. An illustration of this is given by (\ref{pb_diag}).  
\begin{eqnarray} \label{pb_diag}
\xymatrix{\xtild_0 \ar@{->>}[r]^-{\sim} & X_{01} & \xtild_1 \ar@{->>}[l]_-{\sim} \ar@{->>}[r]^-{\sim} & X_{12} & \xtild_2 \ar@{->>}[l]_-{\sim} \\
       &   Y_1 \ar[lu] \ar[u] \ar[ru]  &      &  Y_2  \ar[lu] \ar[u] \ar[ru]  &      \\
				&     &    Z \ar[lu]  \ar[ru]   &   &   }
\end{eqnarray}
%Here $Z$ is the pullback of $Y_1 \lra \xtild_1 \lla Y_2$, which is also the desired limit (this is easy to check by applying the universal property for limits).  
Now applying the fact that the pullback of a fibration is again a fibration, and the fact that the pullback of a weak equivalence along a fibration is again a weak equivalence, we deduce that the map from $\underset{\mcali_E}{\text{lim}} \; RF^{!} \theta^*_{AE}(\xa)$ to each piece of the diagram is a weak equivalence.

\subsection{Proof of the main result of the section}  \label{iso_cof_subsection}

The goal here is to prove Theorem~\ref{iso_cof_thm}, which is the main result of this section.

\begin{proof}[Proof of Theorem~\ref{iso_cof_thm}]
Let $U \hra U'$ be an isotopy equivalence of $\om$, and let $L \colon U \times I \lra U', (x, t) \mapsto L_t(x)$, be an isotopy from $U$ to $U'$.  Our aim  is to show that the canonical map $F^{!}(U') \lra F^{!}(U)$ is a weak equivalence. The idea  is to first consider the commutative diagram (\ref{big_diag}), which will be defined below (for $V \in \ok(U)$).  Next we will show that the map 
\[ 
\underset{V \in \ok(U)}{\text{holim}}\; (f_1) \colon F^{!}(U') \lra \underset{V \in \ok(U)}{\text{holim}} \; B_1(V)
\]
 is a weak equivalence. By the two-out-of-three axioms, we will deduce successively that the maps $\underset{V \in \ok(U)}{\text{holim}}\; (\ftild_1)$, $\underset{V \in \ok(U)}{\text{holim}}\; (g)$, $\underset{V \in \ok(U)}{\text{holim}}\; (\ftild_0),$ and $\underset{V \in \ok(U)}{\text{holim}}\; (f_0)$ are weak equivalences. Using the fact that this latter morphism is a weak equivalence, we will deduce the theorem. 
 
\begin{eqnarray} \label{big_diag}
\xymatrix{               &  &  \bt_0(V) \ar[rrd]^-{k_0}_-{\sim} \ar[d]^-{\sim} &        &     \\
          QF^{!}(U') \ar[rru]^-{\ftild_0} \ar[rrd]^{g} \ar[d]_-{\sim} &   & B_0(V)  &     &      A_0(V)        \\
					F^{!}(U') \ar[rru]_{f_0} \ar[rrd]^{f_1}  &  & G(V)  \ar[rru]^-{\sim}_-{l_0} \ar[rrd]^-{l_1}_-{\sim} &  & \\
					QF^{!}(U') \ar[rru]_{g} \ar[rrd]_-{\ftild_1} \ar[u]^-{\sim} &   &  B_1(V) &  & A_1(V)   \\
							                &   &     \bt_1(V) \ar[u]_-{\sim} \ar[rru]^-{\sim}_-{k_1}   &  &  }
\end{eqnarray}

Now we explain the construction of the diagram (\ref{big_diag}). 

\begin{enumerate}
\item[$\bullet$] Q(-) is the cofibrant replacement functor in $\mcalm$.
\item[$\bullet$] The objects $B_i(V), \bt_i(V), i \in \{0, 1\}$, are defined as
\[
B_i(V) = \underset{(W, \xa) \in \dv}{\text{holim}} \; F^{!}(L_i(W)) \quad \text{and} \quad \bt_i(V) = \underset{(W, \xa) \in \dv}{\text{holim}} \; QF^{!}(L_i(W)).
\] 
(Recall that $L_0 \colon U \lra U'$ is the inclusion functor; so $L_0(W) = W$.) Clearly these objects are functorial in $V$. Indeed, if $V \hra V'$ is a morphism of $\ok(U)$ then we have the inclusion functor $\theta \colon \dv \lra \dvp$ defined in the course of the proof of Proposition~\ref{dv_functor_prop}. This latter functor and Proposition~\ref{induced_holim_propg} allow us to get the desired functoriality. 
\item[$\bullet$] G(V) is defined as $G(V) = \underset{(W, \xa) \in \dv}{\text{holim}} \; H(W, \xa)$, where $H \colon \dv \lra \mcalm$ is the functor from (\ref{functor_h}). As before the construction that sends $V$ to $G(V)$ is a contravariant functor. 
\item[$\bullet$] The maps $f_i, \ftild_i, i \in \{0, 1\}$, are defined as $f_i = \underset{\dv}{\text{holim}} (h_i)$ and $\ftild_i = \underset{\dv}{\text{holim}} (Qh_i)$ where $h_i \colon F^{!}(U') \lra F^{!}(L_i(W))$ is the obvious map obtained by applying $F^{!}$ to the inclusion $L_i(W) \subseteq U'$.  
\item[$\bullet$] The map $g$  is defined as $g = \underset{\dv}{\text{holim}} (h)$, where $h \colon QF^{!}(U') \lra H(W, \xa)$ is the map from (\ref{map_h}). 
\item[$\bullet$] The objects $A_i(V), i \in \{0, 1\}$, are defined as $A_i(V) := \underset{(W,\xa) \in \dv}{\text{holim}} \; P_i(W, \xa),$ where $P_i \colon \dv \lra \mcalm$ comes from (\ref{functor_po}). Again as before $A_i(V)$ is functorial in $V$. 
\item[$\bullet$] The maps $k_i, i \in \{0, 1\}$, are defined as $k_i = \underset{\dv}{\text{holim}} (p_i)$, where $p_i \colon QF^{!}(W) \stackrel{\sim}{\lra} P_i(W, \xa)$ comes from (\ref{qfv_colim}). Since $p_i$ is a weak equivalence, and since by assumption $F$ is objectwise fibrant, it follows that $k_i$ is a weak equivalence as well by Theorem~\ref{fib_cofib_thmg} \footnote{If $P_i(W, \xa)$ is not fibrant, one can always substitute it by its fibrant replacement}.
\item[$\bullet$] Lastly, the maps $l_i, i \in \{0, 1\}$, are defined as $l_i = \underset{\dv}{\text{holim}} (\eta_i)$, where $\eta_i \colon H \stackrel{\sim}{\lra} P_i$ is the natural transformation from (\ref{etai}). As before, $l_i$ is a  weak equivalence.  
\end{enumerate} 	
%Certainly $k_i$ is a weak equivalence since the category $\dv$ is contractible by Proposition~\ref{dv_contractible_prop}, and since $p_i$ is a weak equivalence for every $\xa$ \footnote{A precise reference is needed...}. 
%\item[$\bullet$] Lastly, $l_i, i= 0, 1,$ is defined as $l_i = \underset{\dv}{\text{holim}} \; (\eta_i)$, where $\eta_i \colon H \stackrel{\sim}{\lra} P_i$ is the natural transformation from (\ref{etai}). Since $\eta_i$ is a weak equivalence, and since by assumption every object of $\mcalm$ is fibrant, it follows that $l_i$ is a weak equivalence as well by Theorem~\ref{fib_cofib_thm}.  

Using definitions, one can see that every map from the diagram (\ref{big_diag}) is natural in $V$. One can also check that the squares containing $f_i, \ftild_i, k_i$ and $l_i, i \in \{0, 1\}$, are both commutative. Now  applying the homotopy limit functor (when $V$ runs over $\ok(U)$) to each morphism of  (\ref{big_diag}), we get a new diagram, denoted $\mathbb{D}$, in which the map $\text{holim}(f_1) \colon F^{!}(U') \lra \underset{V \in \ok(U)}{\text{holim}} \; B_1(V)$ is a weak equivalence because of the following. First consider the following commutative diagram constructed as follows. 
\begin{eqnarray} \label{small_diag}
\xymatrix{\underset{W \in \ok(L_1(V))}{\text{holim}}\; F^{!}(W)  \ar[rr]^-{\sim}  &   &  B_1(V) \\
           F^{!}(L_1(V))  \ar[u]^-{\sim} &   &  F^{!}(U'). \ar[u]_-{f_1} \ar[ll]^-{h_1} \ar[llu]_-{q}   }
\end{eqnarray}
\begin{enumerate}
\item[$\bullet$] The top map is nothing but $[\theta_1; F^{!}]$ (see Proposition~\ref{induced_holim_propg} for the notation \lq\lq$[-;-]$\rq\rq{}), where $\theta_1 \colon \dv \lra \ok(L_1(V))$ is defined as $\theta_1(W, \xa) = L_1(W)$, and $F^{!} \colon \ok(L_1(V)) \lra \mcalm$ is just the restriction of $F^{!}$ to $\ok(L_1(V))$. By Corollary~\ref{theta_zo_lem} the functor $\theta_1$ is homotopy right cofinal, and therefore the map $[\theta_1; F^{!}]$ is a weak equivalence by Theorem~\ref{htpy_cofinal_thmg}. 
\item[$\bullet$] The maps $f_1$ and $h_1$ have been defined before, while $q$ is induced by the canonical map $F^{!}(U') \lra F^{!}(W)$. 
\item[$\bullet$] The lefthand vertical map is the map $\text{holim} (\eta)$, where   $ \eta[W] \colon F(W) \stackrel{\sim}{\lra}  F^{!}(W)$ is the map from Proposition~\ref{fsrik_propg}.  
\end{enumerate}
Applying the homotopy limit functor (when $V$ runs over $\ok(U)$ of course) to each morphism of (\ref{small_diag}), we get a new commutative diagram, denoted $\mathbb{S}$, in which the map $\text{holim} (h_1) \colon F^{!}(U') \lra \underset{V \in \ok(U)}{\text{holim}} F^{!}(L_1(V))$ is a weak equivalence because of the following reason. Consider the functor $\theta \colon \ok(U) \lra \ok(U')$ defined as $\theta(V) = L_1(V)$. Also consider $F^{!} \colon \ok(U') \lra \mcalm$. Clearly $\theta$ is an isomorphism since $L_1 \colon U \lra U'$ is a homeomorphism.  So for any $W \in \ok(U')$, the pair $(\theta^{-1}(W), id)$ is the initial object of the under category $W \downarrow \theta$. This shows that $\theta$ is homotopy right cofinal, and therefore the map $[\theta; F^{!}] \colon \underset{V \in \ok(U')}{\text{holim}} \; F^{!}(V) \lra \underset{V \in \ok(U)}{\text{holim}} \; (\theta^{*}F^{!})(V)$ is a weak equivalence by Theorem~\ref{htpy_cofinal_thmg}. By inspection,  the map $\text{holim}(h_1)$ is nothing but the composition
\[
\xymatrix{\underset{V \in \ok(U')}{\text{holim}} \; F(V)  \ar[rr]^-{\text{holim}(\eta)}_-{\sim} &  &  \underset{V \in \ok(U')}{\text{holim}} \; F^{!}(V) \ar[rr]^-{[\theta; F^{!}]}_-{\sim}  &  &   \underset{V \in \ok(U)}{\text{holim}} \; (\theta^{*}F^{!})(V),}
\]
where  $\eta[V] \colon F(V) \stackrel{\sim}{\lra} F^{!}(V)$ is again the map from Proposition~\ref{fsrik_propg}. Now, applying the two-out-of-three axiom to the diagram $\mathbb{S}$ we deduce that the map $\text{holim}(f_1)$ is a weak equivalence. 

We come back to the diagram $\mathbb{D}$. As we said before, the two-out-of-three axiom shows successively that the maps $\text{holim}(\ftild_1)$, $\text{holim}(g), \text{holim} (\ftild_0), $ and $\text{holim} (f_0)$ are weak equivalences. Now, replacing \lq\lq$1 $\rq\rq{}   by  \lq\lq$0$\rq\rq{}   in the diagram (\ref{small_diag}), 
%and using the same reasoning as before, 
one can see that the map $\text{holim}(h_0) \colon F^{!}(U') \lra F^{!}(U)$ is a weak equivalence by the two-out-of-three axiom. But this is what we had to show. 
\end{proof}

\textsf{University of Regina, 3737 Wascana Pkwy, Regina, SK S4S 0A2, Canada\\
Department of Mathematics and Statistics\\}
\textit{E-mail address: pso748@uregina.ca}

\textsf{University of Regina, 3737 Wascana Pkwy, Regina, SK S4S 0A2, Canada\\
Department of Mathematics and Statistics\\}
\textit{E-mail address: donald.stanley@uregina.ca}

\end{document}